\documentclass[12pt,a4paper,reqno]{amsart}
\usepackage{amssymb}
\usepackage{amscd}
\usepackage{enumerate}
\usepackage{graphicx}
\usepackage{siunitx}
\usepackage{tikz-cd}
\usepackage{bm}
\numberwithin{equation}{section}

\usepackage{mathtools}
\usepackage[tableposition=top]{caption}
\usepackage{booktabs,dcolumn}

\DeclareFontFamily{OT1}{rsfs}{}
\DeclareFontShape{OT1}{rsfs}{n}{it}{<-> rsfs10}{}
\DeclareMathAlphabet{\mathscr}{OT1}{rsfs}{n}{it}

\theoremstyle{plain}

\newtheorem{theorem}{Theorem}[section]
\newtheorem{proposition}[theorem]{Proposition}
\newtheorem{lemma}[theorem]{Lemma}
\newtheorem{corollary}[theorem]{Corollary}
\newtheorem{conjecture}[theorem]{Conjecture}

\theoremstyle{definition}

\newtheorem{remark}[theorem]{Remark}

\newtheorem{example}[theorem]{Example}

\renewcommand\P{\mathbb{P}}
\newcommand\E{\mathbb{E}}

\newcommand\R{\mathbb{R}}
\newcommand\Z{\mathbb{Z}}
\newcommand\N{\mathbb{N}}

\newcommand\C{\mathbb{C}}

\renewcommand\Re{{\operatorname{Re}}}
\renewcommand\Im{{\operatorname{Im}}}
\newcommand\eps{\varepsilon}
\newcommand\deriv{\zeta}
\newcommand\zero{\lambda}

\begin{document}

\title[Sendov's conjecture for high degree polynomials]{Sendov's conjecture for sufficiently high degree polynomials}

\author{Terence Tao}
\address{UCLA Department of Mathematics, Los Angeles, CA 90095-1555.}
\email{tao@math.ucla.edu}


\subjclass[2010]{30C15, 31A05}

\begin{abstract}  \emph{Sendov's conjecture} asserts that if a complex polynomial $f$ of degree $n \geq 2$ has all of its zeroes in closed unit disk $\{ z: |z| \leq 1 \}$, then for each such zero $\zero_0$ there is a zero of the derivative $f'$ in the closed unit disk $\{ z: |z-\zero_0| \leq 1 \}$.  This conjecture is known for $n < 9$, but only partial results are available for higher $n$.  We show that there exists a constant $n_0$ such that Sendov's conjecture holds for $n \geq n_0$.  For $\zero_0$ away from the origin and the unit circle we can appeal to the prior work of D\'egot and Chalebgwa; for $\zero_0$ near the unit circle we refine a previous argument of Miller (and also invoke results of Chijiwa when $\zero_0$ is extremely close to the unit circle); and for $\zero_0$ near the origin we introduce a new argument using compactness methods, balayage, and the argument principle.
\end{abstract}

\maketitle


\section{Introduction}

Given a complex number $z_0 \in \C$ and a radius $r>0$, we define the open disk
$$ D(z_0,r) \coloneqq \{ z \in \C: |z-z_0| < r \},$$
the closed disk
$$ \overline{D(z_0,r)} \coloneqq \{ z \in \C: |z-z_0| \leq r \},$$
the exterior region
$$ \overline{D(z_0,r)}^c \coloneqq \{ z \in \C: |z-z_0| > r \},$$
and the circle
$$ \partial  D(z_0,r) \coloneqq \{ z \in \C: |z-z_0| = r \}.$$

The following conjecture was first introduced by Sendov in 1958 (see \cite[p. 267]{marden}), and first reported in the literature by Hayman \cite[Problem 4.5]{hayman} who incorrectly attributed it to Ilieff who was presenting the conjecture:

\begin{conjecture}[Sendov's conjecture]\label{sendov}  Let $f: \C \to \C$ be a polynomial of degree $n \geq 2$ that has all zeroes in the closed unit disk $\overline{D(0,1)}$.  If $\zero_0$ is one of these zeroes, then $f'$ has at least one zero in $\overline{D(\zero_0,1)}$.
\end{conjecture}

There is a long history of partial results towards this conjecture; see for instance \cite{marden-survey}, \cite{sendov}, \cite{sendov-2}, \cite{rahman}, \cite{sheil} for some surveys of results.  The conjecture is known for low degrees, and specifically for all $n < 9$ \cite{brown-xiang}.  Our main result is in the opposite direction, covering the high degree case:

\begin{theorem}[Sendov's conjecture for sufficiently high degree polynomials]\label{main}  Sendov's conjecture is true for all sufficiently large $n$.  That is, there exists an absolute constant $n_0$ such that Sendov's conjecture holds for $n \geq n_0$.
\end{theorem}

For each fixed $n$, Sendov's conjecture is a first-order statement in real arithmetic and is therefore decidable in finite time by Tarski's theorem \cite{tarski}; see the recent thesis \cite{spjut} for an explicit implementation of this decision procedure.  Thus, if we had an explicit upper bound for the quantity $n_0$ in Theorem \ref{main}, then Sendov's conjecture would also be decidable in finite time.  Our proof of Theorem \ref{main} relies on compactness methods and so does not easily give an effective value of $n_0$; however, we believe the use of these compactness methods is an inessential\footnote{For instance, the various appeals in this paper to unique continuation theorems for certain harmonic or holomorphic functions that emerge in an asymptotic limit could be replaced by Carleman inequalities, which are quantitative enough to be applied non-asymptotically to good effect, thus removing the need to appeal to compactness for that stage of the argument at least.} (but convenient) feature of the argument and can be removed to provide an explicit value of $n_0$ and hence ensure the decidability of the conjecture.  However, this value is likely going to be far too large to be able to make the verification of Sendov's conjecture practical\footnote{For instance, the results of Chalebgwa \cite{chalebgwa}, which currently is the only result with explicit bounds that covers the case of intermediate values of $a$, only becomes non-vacuous for $n \geq 2.8 \times 10^7$.} without further ideas.

As is traditional in the literature on this conjecture, we shall normalize $f$ to be monic (by multiplying $f$ by a constant), and also normalize $\zero_0$ to be a real number $a \in [0,1]$ (by applying a rotation around the origin).

Our arguments will use compactness methods (or more precisely ``compactness and contradiction'' methods), so it is convenient to rephrase the above theorem in the following equivalent asymptotic\footnote{Alternatively, one could use the framework of nonstandard analysis and derive a contradiction using a single unbounded (nonstandard) degree $n$ rather than a sequence of (standard) degrees $n$ going off to infinity.  Related to this, one could work with an ultrafilter of sets of natural numbers rather than with a sequence of natural numbers that is repeatedly refined to a subsequence as the argument progresses.  We leave these reformulations of the argument to the interested reader; they alter the notation used in the proofs but do not otherwise make significant changes to the core underlying ideas of those proofs.} contradiction form:

\begin{theorem}[Main theorem, asymptotic contradiction form]\label{main-contr}  Let $n$ range over a sequence of natural numbers going to infinity.  For each $n$ in this sequence, let $f = f^{(n)}$ be a monic polynomial of degree $n$ with all zeroes in $\overline{D(0,1)}$, and let $a = a^{(n)} \in [0,1]$ be such that $f(a)=0$.  Suppose also that for every $n$ in the sequence, $f'$ has no zeroes in $\overline{D(a,1)}$.  Then one can derive a contradiction.
\end{theorem}

It is clear that Theorem \ref{main} implies Theorem \ref{main-contr}.  Conversely, if Theorem \ref{main} failed, then by taking a sequence of normalized counterexamples of increasingly large degree one obtains a sequence of $n$, and a polynomial $f = f^{(n)}$ and real number $a = a^{(n)}$ indexed by this sequence obeying the hypotheses of Theorem \ref{main-contr}, so that the latter theorem would fail also.  This gives the desired equivalence.  For us, the advantages of working in
this asymptotic contradiction framework are that one can take advantage of asymptotic notation such as $o(1)$, and exploit various compactness theorems to extract useful limiting objects as $n \to \infty$ in which all the $o(1)$ errors are eliminated, after passing to a subsequence if necessary.  Such arguments can be viewed as a more flexible form of the more classical ``normal families arguments'' that one sees in the complex analysis literature, in which Montel's theorem is used as the primary source of compactness, but in our framework we are free to appeal to a wider range of compactness theorems (such as Prokhorov's theorem or the Bolzano--Weierstrass theorem) than just Montel's theorem.

We always reserve the right to pass to a subsequence of $n$ as necessary in order to improve the convergence as $n \to \infty$.
For instance, by the Bolzano--Weierstrass theorem, and passing to a subsequence if necessary, we may assume the real numbers $a = a^{(n)} \in [0,1]$ converge to a fixed\footnote{As discussed in Section \ref{notation-sec}, we use ``fixed'' in this paper to refer to any expression that does not depend on the asymptotic parameter $n$.  In the language of nonstandard analysis, $a^{(\infty)}$ is the \emph{standard part} of $a$.} limit $a^{(\infty)} \in [0,1]$.  Using the asymptotic notation conventions from Section \ref{notation-sec}, we thus have
$$ a = a^{(\infty)} + o(1).$$
Henceforth $f, a, a^{(\infty)}$ will always be assumed to obey the above properties.  

At this point it is worth pointing out two key near-counterexamples to Theorem \ref{main-contr}:

\begin{example}[Near-counterexample near unit circle]\label{ex-1}  For each $n \in \N$, set $f(z) \coloneqq z^n - 1$, and $a \coloneqq 1$.  Then all the zeroes of $f$ lie in $\overline{D(0,1)}$, and $f'$ just barely has zeroes in $\overline{D(a,1)}$ since the zeroes are all at the origin which lies on the boundary circle $\partial D(a,1)$.  In this case we have $a^{(\infty)}=1$.
\end{example}

\begin{example}[Near-counterexample near origin]\label{ex-2}  For each $n \in \N$, set $f(z) \coloneqq z^n - z$, and $a \coloneqq 0$.  Then all the zeroes of $f$ lie in $\overline{D(0,1)}$ and $f'$ just barely has zeroes in $\overline{D(a,1)}$ since the zeroes all lie at distance $n^{-\frac{1}{n-1}} = 1 - O(\frac{\log n}{n})$ from $a$.  In this case we have $a^{(\infty)}=0$.
\end{example}

These examples hint that the cases $a^{(\infty)}=0$, $0 < a^{(\infty)} < 1$, and $a^{(\infty)} = 1$ may need to be treated by separate arguments, with the two extremal cases $a^{(\infty)}=0,1$ being the most difficult; as we shall see, the $a^{(\infty)}=1$ case turns out to be particularly delicate.  This is indeed reflected by past work on this area, a selection of which we now present in the language of the asymptotic contradiction framework.  With regards to the $a^{(\infty)}=0$ case, we have the following existing results towards Theorem \ref{main-contr}:

\begin{itemize}
\item The Gauss-Lucas theorem establishes Theorem \ref{main-contr} when $a=0$.  Schmeisser \cite{schmeisser} also established Theorem \ref{main-contr} in this case if one only excludes zeroes of $f'$ in $D(a,1)$ rather than $\overline{D(a,1)}$.
\item From the final displayed equation in page 66 of the survey \cite{bojanov-survey} of Bojanov, Theorem \ref{main-contr} is established when $a \leq \frac{1}{n-1}$.  
\end{itemize}

In the $0 < a^{(\infty)} < 1$ regime (as well as a slight enlargement thereof), we have the following results towards Theorem \ref{main-contr}:

\begin{itemize}
\item D\'egot \cite[Theorem 9]{degot} established Theorem \ref{main-contr} when $0 < a^{(\infty)} < 1$ (or equivalently, if $c-o(1) \leq a \leq 1-c+o(1)$ for some fixed constant $c>0$).
\item By refining D\'egot's methods, Chalebgwa \cite[Theorem 1.1]{chalebgwa} established Theorem \ref{main-contr} if 
\begin{equation}\label{c7}
C n^{-1/7} \leq a \leq 1 - C n^{-1/4}
\end{equation}
for an explicit\footnote{More precisely, Chalebgwa established Theorem \ref{main-contr} whenever $n \geq \frac{20800}{a^7 (1-a)^4}$.} absolute constant $C$.  A somewhat similar result also appears in \cite[p. 217]{sheil} assuming some lower bound on the distance between $a$ and the other zeroes $\lambda$ of $f$.
\end{itemize}

Finally, in the $a^{(\infty)}=1$ regime, the following results were obtained towards Theorem \ref{main-contr}:

\begin{itemize}
\item Rubinstein \cite{rubinstein}, Goodman--Rahman--Ratti \cite{goodman}, and Joyal \cite{joyal} all established Theorem \ref{main-contr} when $a=1$.
\item Miller \cite{miller} and V\^aj\^aitu--Zaharescu \cite{vz} independently improved these results by establishing Theorem \ref{main-contr} when $a \geq 1 - \eps_n$, for some $\eps_n>0$ depending only on $n$.
\item Chijiwa \cite{chijiwa} refined Miller's methods and established Theorem \ref{main-contr} when 
\begin{equation}\label{c8}
a \geq 1-\frac{1}{2n^9 4^n}.
\end{equation}
\item Kasmalkar \cite{kasmalkar} improved the analysis of Chijiwa to establish Theorem \ref{main-contr} when
\begin{equation}\label{c8a}
a \geq 1-\frac{90}{n^{12} \log n}.
\end{equation}

\end{itemize}

These above results, as well as the new results in this paper, are summarized schematically in Figure \ref{fig:ranges}. We also remark that as a consequence of the results of Bojanov, Rahman, and Satti \cite{brs} that Theorem \ref{main-contr} can be established for all values of $a$ if one assumes that $f'$ has no zeroes in the slight enlargement $\overline{D(a, 2^{1/n})}$ of the disk $\overline{D(a,1)}$.  Shrinking the radius $2^{1/n} = 1 + O(1/n)$ here to $1$ turns out to be remarkably challenging, even in the asymptotic regime when $n$ is sufficiently large.

In this paper we will supplement the above partial results towards Theorem \ref{main-contr} with the following two additional results:

\begin{itemize}
\item In Theorem \ref{contra-origin}, we establish Theorem \ref{main-contr} assuming that 
\begin{equation}\label{c9}
a = o(1/\log n)
\end{equation}
for all $n$ in the ambient sequence (so in particular $a^{(\infty)} = 0$).
\item In Theorem \ref{contra-circle}, we establish Theorem \ref{main-contr} assuming that there exists a fixed $\eps_0>0$ such that
\begin{equation}\label{c10}
1-o(1) \leq a \leq 1 - \eps_0^n 
\end{equation}
for all $n$ in the ambient sequence (so in particular $a^{(\infty)}=1$).
\end{itemize}

These two results, combined with the prior result \eqref{c7} of Chalebgwa \cite{chalebgwa} and either the result \eqref{c8} of  Chijiwa \cite{chijiwa} or the result \eqref{c8a} of Kasmalkar \cite{kasmalkar}, furnish a complete proof of Theorem \ref{main-contr} and hence Theorem \ref{main}, since after passing to a subsequence if necessary one can always arrange matters so that one of the four conditions \eqref{c7}, \eqref{c8}, \eqref{c9}, \eqref{c10} (say) holds for all $n$ in the surviving subsequence; see Figure \ref{fig:ranges}.  In fact a modification of the arguments here can in fact cover the entire range $[0,1]$ without the assistance of the past results \eqref{c7}, \eqref{c8}, \eqref{c8a}, thus making the proof of Theorem \ref{main} more self-contained; see Remarks \ref{degot-time}, \ref{low-cover}, \ref{miller-time}.

\begin{figure} [t]
\centering
\includegraphics[width=4in]{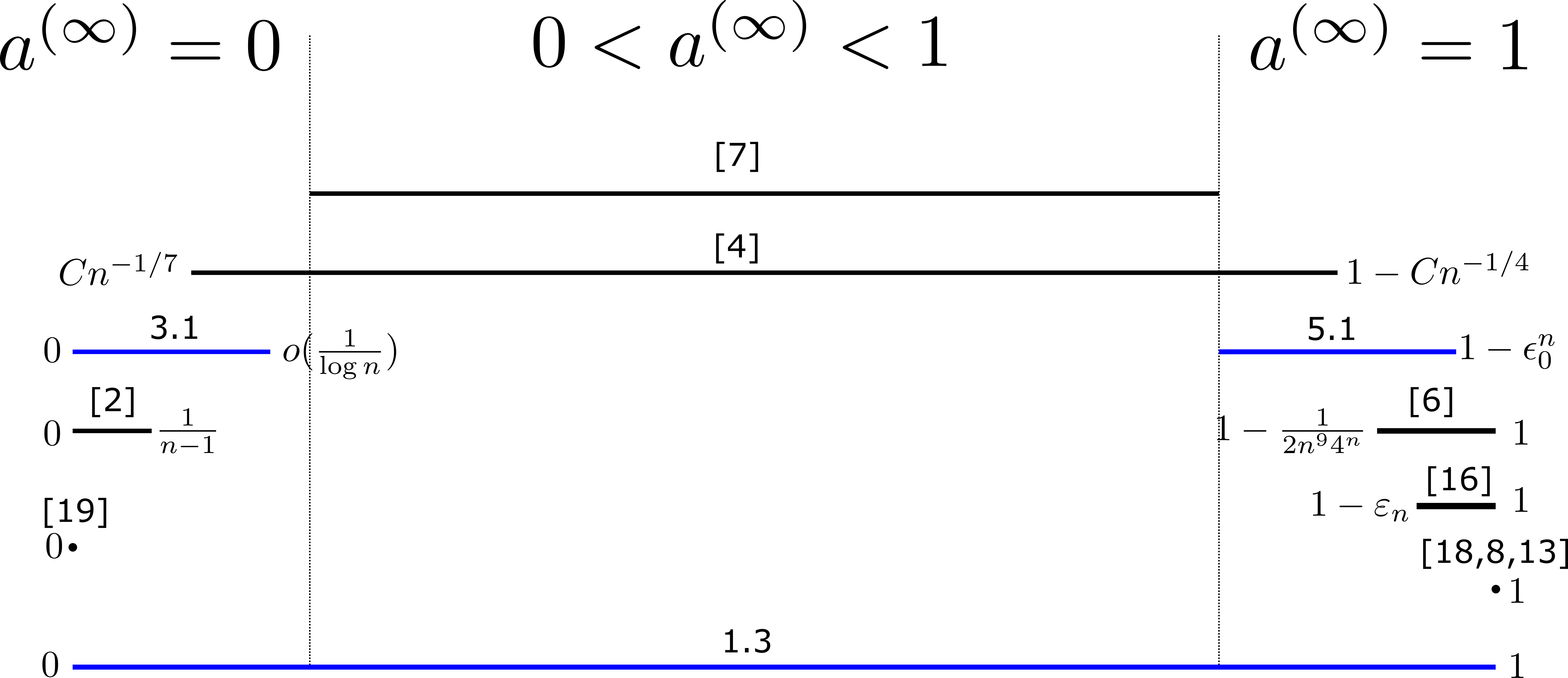}
\caption{A schematic depiction (not drawn to scale) of the ranges of $a$ covered by current (in blue) and past (in black) partial results towards Theorem \ref{main-contr}.  In particular, Theorem \ref{main-contr} follows by combining Theorems \ref{contra-origin}, \ref{contra-circle} with various combinations of the previous work in \cite{degot}, \cite{chalebgwa}, \cite{kasmalkar}, \cite{chijiwa}, \cite{vz}, \cite{miller} or Remarks \ref{degot-time}, \ref{low-cover}, \ref{miller-time}; for instance one can use the results in \cite{chalebgwa}, \cite{chijiwa} to cover all remaining cases.}
\label{fig:ranges}
\end{figure}

To analyze these two new regimes \eqref{c9}, \eqref{c10} it is convenient to use the language of probability theory, particularly in order to access the notion of \emph{convergence in distribution} to extract certain asymptotic profiles.  For each $n$ in the ambient sequence, let $\zero = \zero^{(n)}$ denote a zero of the polynomial $f$, chosen uniformly at random from the $n$ zeroes of $f$ (counting multiplicity), and let $\deriv = \deriv^{(n)}$ denote a random zero of the polynomial $f'$, chosen uniformly at random from the $n-1$ zeroes of $f'$ (again counting multiplicity) and independently of $\zero$.   We permit the underlying probability space $\Omega^{(n)}$ used to define these random variables to vary with $n$, and use $\E = \E^{(n)}$ and $\P = \P^{(n)}$ to denote expectation and probability respectively with respect to this space.  We recall some key convergence concepts regarding convergence of a sequence of complex random variables $X^{(n)}$ (defined on a probability space $\Omega^{(n)}$) to a fixed deterministic limit $c \in \C$ or a fixed random limit $X^{(\infty)}$ (defined on a probability space $\Omega^{(\infty)}$, with expectation $\E^{(\infty)}$ and probability $\P^{(\infty)}$).

\begin{itemize}
\item We say that $X^{(n)}$ \emph{converges uniformly} to $c$ if for every fixed $\eps>0$ one surely has $|X^{(n)}-c| \leq \eps$ whenever $n$ is sufficiently large.
\item We say that $X^{(n)}$ \emph{converges in probability} to $c$ if for every fixed $\eps>0$, $\P^{(n)}( |X^{(n)}-c| \leq \eps ) \to 1$ as $n \to \infty$.
\item We say that $X^{(n)}$ \emph{converges in distribution} to $X^{(\infty)}$ if for every fixed continuous bounded function $\varphi: \C \to \C$, one has $\E^{(n)} \varphi(X^{(n)}) \to \E^{(\infty)} \varphi(X^{(\infty)})$ as $n \to \infty$.  (Equivalently, the law of $X^{(n)}$ converges to the law of $X^{(\infty)}$ in the vague topology.)
\end{itemize}
Note that we do not define a notion of almost sure convergence, since the probability spaces $\Omega^{(n)}$ vary with $n$.  Also observe that if $X^{(\infty)}=c$ is a constant then the notions of convergence in probability and convergence in distribution coincide, but the notion of uniform convergence is strictly stronger. 

It is well known, though usually not expressed in this probabilistic formalism, that the behavior of the polynomial $f$ and its first two derivatives $f', f''$ are closely tied to the behavior of the random variables $\zero, \deriv$.  More precisely, we have the following facts.  Given any complex random variable $\eta$ taking values in a fixed compact set, we define the\footnote{In the literature, the logarithmic potential and Stieltjes transform are often written as functions of the law $\mu_\eta$ of the random variable $\eta$ rather than the random variable itself, but we have chosen in this paper to adopt a probabilistic viewpoint rather than a measure-theoretic one, and are thus downplaying the role of the probability measure $\mu_\eta$.} \emph{logarithmic potential} $U_\eta: \C \to \R$ by
\begin{equation}\label{pot-def}
 U_\eta(z) \coloneqq \E \log \frac{1}{|z-\eta|}
\end{equation}
as well as the closely related \emph{Stieltjes transform} (or \emph{Cauchy transform}) $s_\eta: \C \to \C$ by
\begin{equation}\label{cauchy-def}
 s_\eta(z) \coloneqq \E \frac{1}{z-\eta};
\end{equation}
these functions are well-defined for almost every $z$ (and are locally integrable).  We remark that the Stieltjes transform $s_\eta$ and the law $\mu_\eta$ of $\eta$ are essentially the gradient and Laplacian respectively of the logarithmic potential $U_\eta$, in the sense that
\begin{equation}\label{sz-u}
 s_\eta(x+iy) = - \frac{\partial}{\partial x} U_\eta(x+iy) + i \frac{\partial}{\partial y} U_\eta(x+iy) 
\end{equation}
and
\begin{equation}\label{mu-u}
\mu_\eta = \frac{-1}{2\pi} \left(\frac{\partial^2}{\partial x^2} + \frac{\partial^2}{\partial y^2}\right) U_\eta = \frac{1}{2\pi} \left(\frac{\partial}{\partial x} + i \frac{\partial}{\partial y}\right) s_\eta
\end{equation}
(in a distributional sense at least).  In particular, away from the essential range of $\eta$ (i.e., the support of $\mu_\eta$), $U_\eta$ is harmonic and $s_\eta$ is holomorphic. The logarithmic potential $U_\eta$ is also superharmonic on all of $\C$, but we will not directly utilise this fact.

\begin{lemma}[Basic relations between $\zero,\deriv$ and $f$]\label{basic}\ 
\begin{itemize}
\item[(i)]  (Support) $\zero$ surely lies in the disk $\overline{D(0,1)}$, and $\deriv$ surely lies in the lune $\overline{D(0,1)} \backslash \overline{D(a,1)}$.
\item[(ii)]  (Matching mean) One has
\begin{equation}\label{mean}
\E \zero = \E \deriv.
\end{equation}
\item[(iii)]  (Logarithmic potential) For almost all $z$, one has
\begin{equation}\label{lf}
U_\zero(z) = - \frac{1}{n} \log |f(z)|
\end{equation}
and
\begin{equation}\label{lfp}
U_\deriv(z) = \frac{\log n}{n-1} - \frac{1}{n-1} \log |f'(z)|.
\end{equation}
\item[(iv)]  (Logarithmic derivative) For almost all $z$, one has
\begin{equation}\label{logderiv-f}
s_\zero(z) = \frac{1}{n} \frac{f'(z)}{f(z)}
\end{equation}
and
\begin{equation}\label{logderiv-fp}
s_\deriv(z) = \frac{1}{n-1} \frac{f''(z)}{f'(z)}.
\end{equation}
\item[(v)] (Relation between $\zero$ and $\deriv$) For almost all $z$, one has
\begin{equation}\label{uzs}
 U_{\zero}(z) - \frac{n-1}{n} U_{\deriv}(z) = \frac{1}{n} \log \left|s_\zero(z)\right|
\end{equation}
and
\begin{equation}\label{sz-mod}
 s_{\zero}(z) - \frac{n-1}{n} s_{\deriv}(z) = -\frac{1}{n} \frac{s'_\zero(z)}{s_\zero(z)}
\end{equation}
\item[(vi)] (Integrated log-derivative)  One has
\begin{equation}\label{integ-f}
f(\gamma(1)) = f(\gamma(0)) \exp\left( n \int_\gamma s_\zero(z)\ dz \right)
\end{equation}
for any contour $\gamma \colon [0,1] \to \C$ that avoids the zeroes of $f$, and similarly
\begin{equation}\label{integd-f}
f'(\gamma(1)) = f'(\gamma(0)) \exp\left( (n-1) \int_\gamma s_\deriv(z)\ dz \right)
\end{equation}
for any contour $\gamma \colon [0,1] \to \C$ that avoids the zeroes of $f'$.
\end{itemize}
The identities in (ii)-(vi) are valid for all monic polynomials $f$ of degree $n$ (not just those obeying the hypotheses of Theorem \ref{main-contr}).
\end{lemma}

\begin{proof}  Part (i) is clear from the hypotheses on $f$ and the Gauss--Lucas theorem.  For parts (ii)-(iv), we use the fundamental theorem of algebra to write
$$ f(z) = (z-\zero_1) \dots (z-\zero_n)$$
and
$$ f'(z) = n (z-\deriv_1) \dots (z-\deriv_{n-1})$$
where $\zero_1,\dots,\zero_n$ (resp. $\deriv_1,\dots,\deriv_{n-1}$) are the zeroes of $f$ (resp. $f'$) counting multiplicity.  Comparing the top two leading coefficients of these polynomials gives \eqref{mean}; taking logarithms of the magnitudes gives \eqref{lf}, \eqref{lfp}; and taking logarithmic derivatives (or \eqref{sz-u}) gives \eqref{logderiv-f}, \eqref{logderiv-fp}.  By combining \eqref{lf}, \eqref{lfp}, \eqref{logderiv-f} we obtain \eqref{uzs}, and by combining \eqref{logderiv-f}, \eqref{logderiv-fp}, and the derivative of \eqref{logderiv-f} using the quotient rule we obtain \eqref{sz-mod} (alternatively, one could apply the Cauchy--Riemann operator $\frac{\partial}{\partial x} + i \frac{\partial}{\partial y}$ to \eqref{uzs} and use \eqref{sz-u}).  Finally, one obtains \eqref{integ-f} by viewing \eqref{logderiv-f} as an ODE for the function $f$ and solving via the method of integrating factors; alternatively, one can partition $\gamma$ into small contours on which one has a branch of $\log f$ and apply the fundamental theorem of calculus.  The identity \eqref{integd-f} is established similarly.
\end{proof}

\begin{remark} By using Grace's theorem \cite{grace} (in the special case given in \cite[Satz 5]{szego}) one can obtain further ``zero-free regions'' for $f$ that allow one to reduce the size of the support of $\zero$ to a region smaller than $\overline{D(0,1)}$; see for instance
\cite[Theorem 4]{degot} for an example of such a zero-free region.  Some quantitative zero-free regions of this nature will be employed later in our arguments (notably in \eqref{similar}).  As already noted and exploited by many previous authors (e.g., \cite[Remark 3]{degot}), the zeroes $\zero$ of $f$ will in fact tend to cluster near the boundary $\partial D(0,1)$ of the disk, while the zeroes $\deriv$ of $f'$ will tend to cluster near the right boundary $\overline{D(0,1)} \cap \partial D(a,1)$ of the lune; we will quantify this phenomenon in Theorem \ref{limit-thm} below.
\end{remark}

\begin{remark}  In the limit $n \to \infty$, the equation \eqref{sz-mod} becomes (formally, at least) a partial differential equation for the Stieltjes transform $s_\zero(z)$ under the action of repeated differentiation $f \mapsto f^{(k)}$ of the polynomial, with the renormalised order of differentiation $k/n$ playing the role of a time variable.  See the recent papers \cite{steiner}, \cite{kab} for further discussion.  The same equation also arises in free probability and random matrix theory in the context of the minor process (or the operation of fractional free convolution powers); see \cite{dima}.  However, we will not use the equation \eqref{sz-mod} directly, preferring to work instead with its integrated form \eqref{uzs}.
\end{remark}

We will be able to characterize the limiting behavior of the random variables $\zero, \deriv$, or equivalently the asymptotic distribution of the zeroes of $f,f'$.  By Prokhorov's theorem, we know that after passing to a subsequence, the random variables $\zero, \deriv$ converge in distribution to (fixed) random variables $\zero^{(\infty)}, \deriv^{(\infty)}$ respectively on some probability space $\Omega^{(\infty)}$ taking values in $\overline{D(0,1)}$ respectively, thus
\begin{equation}\label{convert}
 \E \varphi(\zero) = \E \varphi(\zero^{(\infty)}) + o(1); \quad \E \varphi(\deriv) = \E \varphi(\deriv^{(\infty)}) + o(1)
\end{equation}
for any fixed continuous function $\varphi: \C \to \C$.  (One could also ensure that $\zero^{(\infty)}, \deriv^{(\infty)}$ are independent of each other, but we will not need to do so here.)  In view of Lemma \ref{basic}, the random variables $\zero^{(\infty)}, \deriv^{(\infty)}$ carry a lot of information about the ``macroscopic'' behaviour of $f$ and its first two derivatives in the asymptotic limit $n \to \infty$.  For instance $\frac{1}{n} \log \frac{1}{|f(z)|}$ (resp. $\frac{1}{n-1} \log \frac{n}{|f'(z)|}$) will converge locally in $L^1(\C)$ to the logarithmic potential $U_{\zero^{(\infty)}}$ (resp. $U_{\deriv^{(\infty)}}$); see Section \ref{limit-sec}.  One can view the random variables $\zero^{(\infty)}, \deriv^{(\infty)}$ as an abstract ``limit'' of the increasingly high degree polynomials $f = f^{(n)}$, analogously to how graphons can be viewed as limits of sequences of increasingly large graphs \cite{lovasz}.

\begin{example} In Example \ref{ex-1}, $\zero^{(\infty)}$ will be uniformly distributed on the unit circle $\partial D(0,1)$, and $\deriv^{(\infty)}$ will equal $0$ almost surely.  In Example \ref{ex-2}, $\zero^{(\infty)}, \deriv^{(\infty)}$ will both be uniformly distributed on the unit circle $\partial D(0,1)$.
\end{example}

Henceforth we pass to a subsequence so that the limiting random variables $\zero^{(\infty)}, \deriv^{(\infty)}$ exist.  The first main step in our arguments is to obtain a fair amount of control on these limiting variables in the endpoint cases $a^{(\infty)}=0,1$, as well as some related estimates:

\begin{theorem}[Limiting measures]\label{limit-thm}\ 
\begin{itemize}
\item[(i)]  If $a^{(\infty)}=0$, then $\zero^{(\infty)}, \deriv^{(\infty)}$ almost surely lie in the semicircle $C \coloneqq \{ e^{i\theta}: \pi/2 \leq \theta \leq 3\pi/2\}$ and have the same distribution. In particular, $\mathrm{dist}(\zero,C)$ converges in probability to zero. In fact, if $K$ is any fixed compact subset of $\overline{D(0,1)} \backslash C$, then we have the more quantitative estimate
\begin{equation}\label{pzeta}
\P( \zero \in K ) \lesssim a + \frac{\log n}{n^{1/3}}
\end{equation}
(see Section \ref{notation-sec} for our conventions on asymptotic notation).
\item[(ii)]  If $a^{(\infty)}=1$, then $\zero^{(\infty)}$ is uniformly distributed on the unit circle $\partial D(0,1)$, and $\deriv^{(\infty)}$ is almost surely zero.  In particular, $|\zero|$ converges in probability to $1$ and $\deriv$ converges in probability to zero.  In fact we have the more quantitative estimates
\begin{equation}\label{zetas}
 \E \log \frac{1}{|\zero|}, \E \log |\deriv - a| \lesssim \frac{1}{n}.
\end{equation}
\end{itemize}
\end{theorem}

We prove this theorem in Section \ref{limit-sec}.  Our main tools are the identities in Lemma \ref{basic}(iii)-(v) and some potential theory, relying in particular on the unique continuation property of harmonic functions and the basic properties of balayage \cite{gustafsson}; for part (ii) we also use some arguments of D\'egot \cite{degot}.  In particular, we rely heavily on the observation (from \eqref{uzs}) that $\zero$ and $\deriv$ ``look the same from outside $\overline{D(0,1)}$'' in the sense that their balayages to any larger disk $\overline{D(0,R)}$ are very nearly equal.  Theorem \ref{limit-thm} can be compared with \cite[Remark 3]{degot}, which in our language asserts in the $0 < a^{(\infty)} < 1$ case that $|\zero^{(\infty)}| = |\deriv^{(\infty)}-a|=1$ almost surely (or equivalently, that $|\zero|$ and $|\deriv - a|$ both converge in probability to $1$).  Theorem \ref{limit-thm}(ii) is also anticipated in previous work on the $a^{(\infty)}=1$ case; see for instance \cite[p. 633, Lemma 5]{miller} or \cite{chijiwa-disk} for variants of this assertion when $a$ is extremely close to $1$.  There are also analogues in the small $n$ theory; see for instance \cite[Lemma 3.2]{brown-xiang}.  Theorem \ref{limit-thm}(ii) is also consistent with the near-counterexample in Example \ref{ex-1}, as well as the generalisations of this example considered in Section \ref{nearer} below.   Theorem \ref{limit-thm}(i), in contrast, is not compatible with the near-counterexample in Example \ref{ex-2}: the reason being that deleting $\overline{D(a,1)}$ from the unit circle $\partial D(0,1)$ will essentially restrict that circle to the semicircle $C$ when $a$ is small, whereas deleting the slightly smaller disk $D(a,1-\eps)$ (which is the type of region one can delete in the near-counterexample in Example \ref{ex-1}) does not have the same effect.  The error $\frac{\log n}{n^{1/3}}$ term in \eqref{pzeta} could perhaps be improved with more effort, but for our purposes any bound of the form $o(1/\log n)$ will suffice.

In the regime $a = o(1/\log n)$, Theorem \ref{limit-thm}(i) when combined with Theorem \ref{basic}(iv) and the unique continuation properties of holomorphic functions will allow one to obtain enough control on the log-derivative $\frac{f'}{f}$ near the origin that one can establish via the argument principle that the number of zeroes of this log-derivative near this origin, minus the number of poles, is non-negative; but by hypothesis $f'$ will have no zeroes near the origin and $f$ has a zero at $a$, giving the required contradiction.  See Section \ref{origin-sec} for the rigorous details.

In the regime when $1-o(1) \leq a \leq 1 - \eps_0^n$, we use a variant of the analysis of Miller \cite{miller}.  An instructive model case of a near-counterexample in this regime, generalising Example \ref{ex-1}, is presented in Section \ref{nearer}, in which $f'$ has a bounded number of zeroes away from the origin and most zeroes very close to the origin, while $f$ has all of its zeroes near the unit circle.  By performing a suitable Taylor expansion, one finds that the large zeroes $\lambda_j$ of $f'$ stay close to the arc $\overline{D(0,1)} \cap \partial D(1,1)$, but also their second moment $\sum_j \lambda_j^2$ must basically vanish.  These two claims are incompatible (except when all the $\lambda_j$ vanish), thanks to the key geometric observation that non-zero points on the arc $\overline{D(0,1)} \cap \partial D(1,1)$ have argument in $[\pi/3,\pi/2]$ or $[-\pi/2,-\pi/3]$ (see Figure \ref{fig:key}), so that their square will have argument in $[2\pi/3,4\pi/3]$ and hence have negative real part.  It turns out that this analysis can be extended to more general polynomials.  The starting point is to use Theorem \ref{limit-thm}(ii) to get good control on the Stieltjes transform $s_\deriv$, which by Lemma \ref{basic}(vi) then gives good asymptotics for $f'$.  Applying the fundamental theorem of calculus, one then gains enough control on $f$ to locate the zeroes $\zero$ of this function rather precisely, at which point one can repeat the analysis used to rule out the counterexamples studied in Section \ref{nearer}.  The rigorous details of this argument are presented in Section \ref{circle-sec}. 

\begin{figure} [t]
\centering
\includegraphics[width=4in]{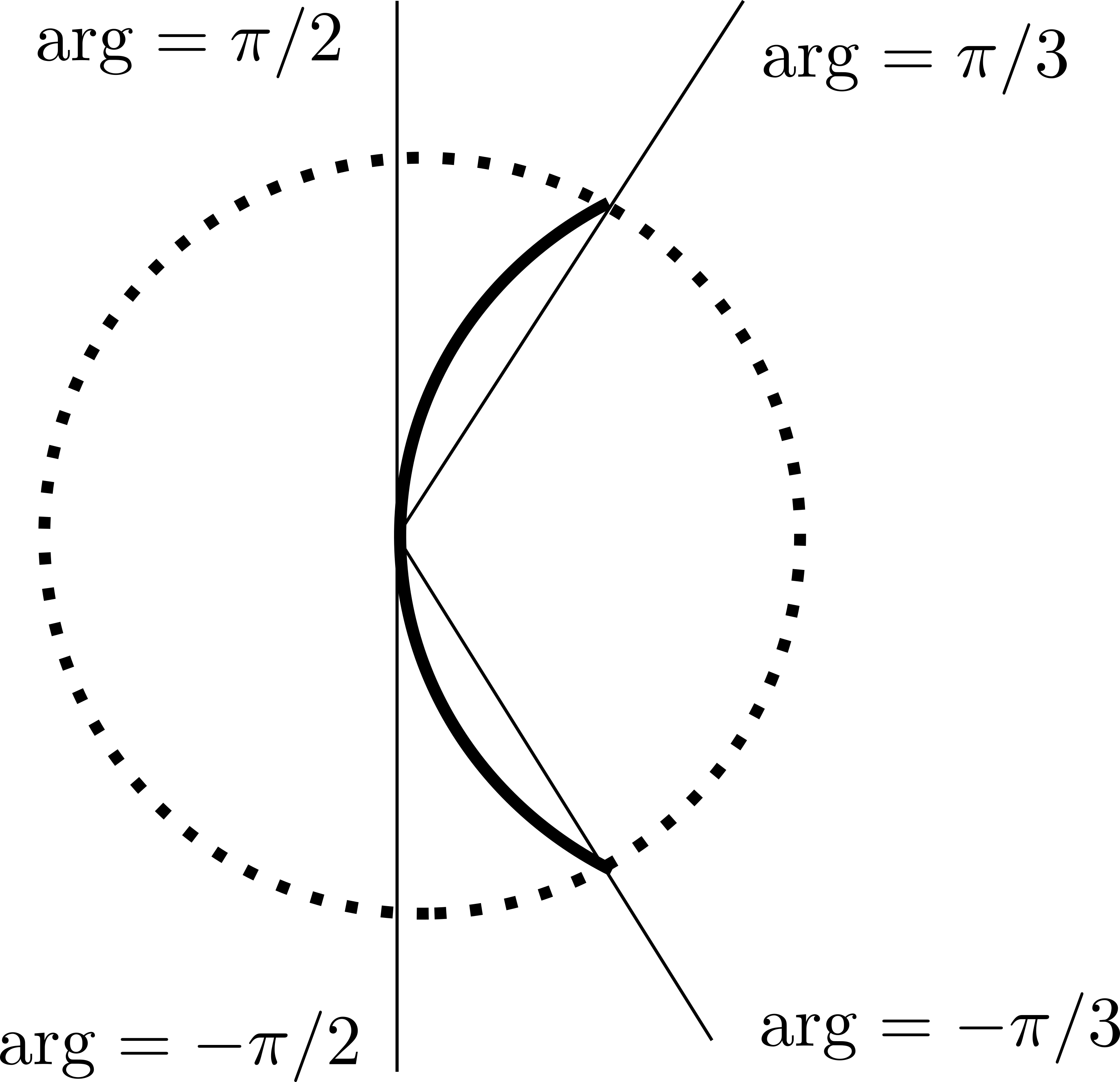}
\caption{Non-zero points on the arc $\overline{D(0,1)} \cap \partial D(1,1)$ have argument in $[\pi/3,\pi/2]$ or $[-\pi/2,-\pi/3]$. In particular, points on this arc subtend an angle of less than $\pi/4$ with the imaginary axis at the origin, so that their square lies in the left half-plane.}
\label{fig:key}
\end{figure}

\subsection{Acknowledgments}

The author is supported by NSF grant DMS-1764034 and by a Simons Investigator Award. We thank Alexandre Eremenko for a correction, and the anonymous referee for many helpful suggestions and corrections.

\subsection{Asymptotic notation}\label{notation-sec}

The following (somewhat nonstandard\footnote{Indeed, as remarked previously, the notations here can be viewed as a ``cheap'' version of nonstandard analysis; for instance, the objects that we call ``fixed'' here would be referred to as ``standard'' in the nonstandard analysis formalism, quantities that are $o(1)$ here would be referred to as ``infinitesimal'', and so forth.  See \cite{tao-cheap} for further discussion.}) asymptotic notations will be in force throughout this paper:

\begin{itemize}
\item Throughout this paper we assume that $n$ is a parameter ranging in some sequence of natural numbers tending to infinity, which we will call the ``ambient sequence''.  We do not refer to the ambient sequence by a specific mathematical symbol, because we reserve the right to refine the ambient sequence to a subsequence at various stages of the argument (for instance, in order to exploit compactness to improve convergence of some quantity in the limit $n \to \infty$).  Some quantities $X$ in this paper will be independent of $n$ and will be declared to be \emph{fixed}, but quantities $X$ that are not explicitly declared to be fixed are permitted to depend on $n$, thus $X = X^{(n)}$ in this case.  We usually omit the superscript $(n)$ unless we wish to emphasise the dependency on $n$. 
\item We use $X = o(Y)$ to denote the estimate $|X| \leq c(n) Y$ for some quantity $c(n)$ that can depend on $n$ and on fixed quantities, such that $c(n) \to 0$ as $n$ goes to infinity along the ambient sequence (keeping all fixed quantities constant).  Of course, such a relation $X=o(Y)$ will continue to hold if we pass to a subsequence.  If a mathematical statement does not explicitly involve the limit $n \to \infty$, it is understood to hold for all $n$ in the ambient sequence (and hence on all subsequences), unless otherwise specified. 
\item We use $X = O(Y)$, $Y \gtrsim X$, or $X \lesssim Y$ to denote the estimate $|X| \leq CY$ for some quantity $C$ that is fixed.  In particular, $C$ can depend on other fixed quantities, but must remain uniform with respect to quantities that vary with $n$.
\end{itemize}

The following examples will help illustrate these asymptotic conventions:

\begin{example}  Let $f: \C \to \C$ be a function, and let $K$ be a fixed subset of $\C$.  We did not declare $f$ to be fixed, hence $f = f^{(n)}$ is permitted to depend on $n$, which ranges in some (unspecified) ambient sequence.  The assertion that
$$ f(z) = o(1)$$
for all $z \in K$ is equivalent to the sequence $f = f^{(n)}$ converging to zero uniformly in $K$ as $n$ goes to infinity along the ambient sequence.  Indeed, if uniform convergence failed, then one could select\footnote{We freely use the axiom of choice in this paper, though with a little effort it should be possible to structure the argument so that the axiom of choice is not required.} an element $z = z^{(n)}$ of $K$ for each $n$ in the ambient sequence such that $f(z)$ did not go to zero as $n$ went to infinity along the sequence, so that the assertion $f(z) = o(1)$ would fail for this (non-fixed) choice of $z$; conversely if uniform convergence held then one certainly has $f(z)=o(1)$ for all $z \in K$.  On the other hand, if one only asserts that $f(z)=o(1)$ for all \emph{fixed} $z \in K$, this is equivalent to the weaker assertion that $f = f^{(n)}$ converges \emph{pointwise} to zero in $K$, as $n$ goes to infinity along the ambient sequence.  In a similar spirit:
\begin{itemize}
\item $f^{(n)}$ is uniformly bounded on $K$ if and only if\footnote{In particular, given a bound of the form $f(z)=O(1)$ for all $z \in K$, the implied constant in the $O(1)$ notation is automatically uniform in $z$, since otherwise one could use countable choice to locate $z = z^{(n)} \in K$ for which $f(z) = f^{(n)}(z^{(n)})$ was unbounded in $n$.  This is an example of the \emph{countable saturation} (or more precisely $\aleph_0$-saturation) property in model theory.} $f(z) = O(1)$ for all $z \in K$.  
\item $f^{(n)}$ is pointwise bounded on $K$ if and only if $f(z) = O(1)$ for all \emph{fixed} $z \in K$.
\item $f^{(n)}$ converges locally uniformly to zero if and only if, for each fixed compact $K \subset \C$, one has $f(z) = o(1)$ for all $z \in K$.  (Note that this is weaker than asserting that $f(z)=o(1)$ for all $z \in \C$, because $K$ has to be fixed, and the decay rate $c(n)$ implicit in the $o(1)$ notation can depend on this fixed $K$.)
\end{itemize}
\end{example}

\begin{example}  If $a = a^{(n)}$ lies in $[0,1]$, then after passing to a subsequence if necessary, exactly one of the assertions
$$ a = o(1/n)$$
and
$$ a \gtrsim 1/n$$
hold.  Indeed, it is clear that these two assertions cannot simultaneously be true.  If for every fixed $\eps>0$ there exists an $n$ for which $a \leq \eps/n$, then by sending $\eps$ to zero one can extract a subsequence for which $a=o(1/n)$; however if there does not exist such an $n$ for at least one fixed $\eps > 0$, then we instead have $a \gtrsim 1/n$.
\end{example}

\begin{example}[Underspill principle]\label{underspill} If $X = X^{(n)}$ is a complex number, then one has $X \lesssim n^\eps$ for every fixed $\eps>0$ if and only if $X \lesssim n^{o(1)}$.  Indeed, if $X \lesssim n^\eps$ for every fixed $\eps>0$, then one can find an increasing sequence $n_1 < n_2 < \dots$ of natural numbers such that $X \leq n^{1/k}$ whenever $n \geq n_k$.  If we then define $\eps_* = \eps_*^{(n)}$ to be $1/k$ if $k$ is the largest natural number with $n \geq n_k$, or $\delta=1$ if no such $k$ exists, then $\eps_*=o(1)$ and $X \lesssim n^{\eps_*}$, hence $X \lesssim n^{o(1)}$.  The converse implication is also easily established.  We refer to this type of argument in the sequel as ``letting $\eps$ tend to zero sufficiently slowly'' or an application of the ``underspill principle''.
\end{example}

\begin{example}\label{unif}  Let $\deriv$ be a random variable taking values in $\overline{D(0,1)}$.  The assertion that $\deriv$ converges to zero in probability (or in distribution) is equivalent to the assertion that $\P( \deriv \in K ) = o(1)$ for every fixed compact subset $K$ of $\overline{D(0,1)} \backslash \{0\}$.  By the underspill principle, it is also equivalent to the existence of an $\eps = o(1)$ such that $\P( |\deriv| \leq \eps ) = 1-o(1)$.  In contrast, the stronger assertion that $\deriv$ converges to zero \emph{uniformly} is equivalent to the existence of an $\eps = o(1)$ such that $|\deriv| \leq \eps$ holds surely (not just with probability $1-o(1)$).
\end{example}

\section{Control of limiting random variables}\label{limit-sec}

We now prove Theorem \ref{limit-thm}.  We begin with some general potential-theoretic considerations that apply to both parts of the theorem.  Given an almost surely bounded complex random variable $\eta$, we see from Young's inequality that the logarithmic potential $U_\eta$ (defined in \eqref{pot-def}) is defined almost everywhere as a locally integrable function, and is also harmonic outside of the essential range of $\eta$.  If $\eta^{(n)}$ is a sequence of complex random variables taking values in a fixed compact set that converge in distribution to another complex random variable $\eta^{(\infty)}$, then as is well known $U_{\eta^{(n)}}$ converges locally in $L^1(\C)$ to $U_{\eta^{(\infty)}}$ (this can be seen for instance by approximating the Newton potential $\log \frac{1}{|z|}$ by a continuous function plus an error small in $L^1(\C)$).   Here of course we equip the complex plane $\C$ with Lebesgue measure. In particular we have
$$ U_\zero(z) \to U_{\zero^{(\infty)}}(z)$$
and
$$ U_\deriv(z) \to U_{\deriv^{(\infty)}}(z)$$
locally in $L^1(\C)$ as $n \to \infty$. 

For $z$ in the exterior region $\overline{D(0,1)}^c$, we have $z-\zero \in D(z,1)$, which is a disk with a diameter that is a rotation of $[|z|-1, |z|+1]$ around the origin.  Applying M\"obius inversion, we conclude that $\frac{1}{z-\zero}$ takes values in a disk with a diameter that is a rotation of $[\frac{1}{|z|+1}, \frac{1}{|z|-1}]$ around the origin. By convexity, the same holds for the Stieltjes transform $s_\zero(z) =  \E \frac{1}{z-\zero}$.  Applying Lemma \ref{basic}(v), we conclude the inequalities
\begin{equation}\label{umu}
 -\frac{\log(|z|+1)}{n} \leq U_{\zero}(z) - \frac{n-1}{n} U_{\deriv}(z) \leq -\frac{\log(|z|-1)}{n}
\end{equation}
for almost all $z \in \overline{D(0,1)}^c$; in fact this holds for \emph{all} $z \in \overline{D(0,1)}^c$ since both sides are continuous in this region.  Taking limits as $n \to \infty$, we conclude that
\begin{equation}\label{mun}
 U_{\zero^{(\infty)}}(z) = U_{\deriv^{(\infty)}}(z)
\end{equation}
for almost all $z \in \overline{D(0,1)}^c$; again, by continuity this extends to all $z \in \overline{D(0,1)}^c$.

At this point we could proceed either via potential theory or via Fourier analysis; we shall adopt the latter approach.  If we write $z$ in polar form as $z = R e^{i\theta}$ for some $R>1$ and $\theta \in \R$, we see from the Taylor expansion 
\begin{equation}\label{Taylor}
\log \frac{1}{|Re^{i\theta} - w|} = - \log R + \frac{1}{2} \sum_{m \in \Z: m \neq 0} \frac{e^{im\theta} w^{|m|}}{|m| R^{|m|}}
\end{equation}
(valid for all $w \in D(0,R)$) and the Fubini--Tonelli theorem that
\begin{equation}\label{fourier}
U_{\zero^{(\infty)}}(Re^{i\theta}) = - \log R + \frac{1}{2} \sum_{m \in \Z: m \neq 0} \frac{e^{im\theta}}{|m| R^{|m|}} \E (\zero^{(\infty)})^{|m|}.
\end{equation}
Similarly for $\deriv^{(\infty)}$.  From \eqref{mun} and Fourier uniqueness we conclude that the random variables $\zero^{(\infty)}, \deriv^{(\infty)}$ have matching moments, thus
\begin{equation}\label{moment-match}
\E (\zero^{(\infty)})^m = \E (\deriv^{(\infty)})^m
\end{equation}
for all natural numbers $m$.

It is convenient to encode this moment matching relation in terms of Poisson kernels and balayage, in order to exploit the non-negative nature of the Poisson kernel.  If $R>0$ and $w = re^{i\alpha} \in D(0,R)$, let
 $P^R_w: \partial D(0,R) \to \R$ denote the \emph{Poisson kernel}
\begin{equation}\label{poisson}
\begin{split}
 P^R_w(Re^{i\theta}) &\coloneqq \Re \frac{1 + \frac{r}{R} e^{i(\theta-\alpha)}}{1 - \frac{r}{R} e^{i(\theta-\alpha)}} \\
&= \frac{1 - (r/R)^2}{1 - 2(r/R) \cos (\theta-\alpha) + (r/R)^2} \\
&= \sum_{m=-\infty}^\infty (r/R)^{|m|} e^{im(\theta-\alpha)}.
\end{split}
\end{equation}
One can interpret $P^R_w(Re^{i\theta})$ as the normalized probability density of the harmonic measure of $w$ on the circle $\partial D(0,R)$, or equivalently of the first location at which a Brownian motion originating at $w$ exits the disk $\overline{D(0,R)}$.
If $\eta$ is random variable taking values in $D(0,R)$, define the \emph{balayage} $\mathrm{Bal}_R(\eta) \colon \partial D(0,R) \to \R$ by the formula
$$ \mathrm{Bal}_R(\eta) (Re^{i\theta}) \coloneqq \E P^R_\eta( Re^{i\theta} ).$$
One can interpret $\mathrm{Bal}_R(\eta)$ as the normalized probability density of the first location where a Brownian motion originating at $\eta$ exits the disk $\overline{D(0,R)}$.  We refer the reader to \cite{gustafsson} for a further discussion of the classical and modern theory of balayage.  Note that
\begin{equation}\label{mas}
\mathrm{Bal}_R(\eta) (Re^{i\theta}) = 1 + 2 \Re \sum_{m=1}^\infty R^{-m} e^{im\theta} \E \eta^m.
\end{equation}
In particular, from \eqref{moment-match} one has
\begin{equation}\label{bal-bal}
\mathrm{Bal}_R(\zero^{(\infty)}) = \mathrm{Bal}_R(\deriv^{(\infty)}) 
\end{equation}
for any $R > 1$.  Informally, \eqref{bal-bal} asserts that two Brownian motions originating from $\zero^{(\infty)},\deriv^{(\infty)}$ respectively ``look the same'' once they exit the unit disk $\overline{D(0,1)}$ for the first time.  We also remark that $\mathrm{Bal}_R(\eta)$ is related to the Stieltjes transform $s_\eta$ via the identity
$$ \mathrm{Bal}_R(\eta)(Re^{i\theta}) = 2 \mathrm{Re} Re^{i\theta} s_\eta(Re^{i\theta}) - 1,$$
as is easily verified using \eqref{mas}.

There is also a non-asymptotic version of these identities.  If $1 < R \leq 3/2$, then we have $\log \frac{1}{|z-\deriv|} \lesssim \log \frac{1}{R-1}$ for any $z \in \partial D(0,R)$, hence $U_{\deriv}(z) \lesssim \log \frac{1}{R-1}$. From \eqref{umu} we conclude that
$$ U_{\zero}(z) = U_{\deriv}(z) + O\left( \frac{1}{n} \log \frac{1}{R-1} \right)$$
for $z \in \partial D(0,R)$. Extracting Fourier coefficients using \eqref{fourier} we see that
$$
\E \zero^m = \E \deriv^m + O\left( \frac{m R^m}{n} \log \frac{1}{R-1}\right)$$
for any natural number $m$ and $1 < R \leq 3/2$.  We can optimise in $R$ by choosing $R \coloneqq 1 + \frac{1}{m}$ to obtain
$$
\E \zero^m = \E \deriv^m + O\left( \frac{m \log m}{n} \right)$$
(the case $m=1$ following from Lemma \ref{basic}(ii)).  Combining this with \eqref{mas} and the bound
$$ \sum_{m=1}^\infty R^{-m} m \log m \lesssim \frac{\log \frac{1}{R-1}}{(R-1)^2} $$
for any $1 < R \leq 3/2$ (which is easily established by dyadic decomposition, with the terms $m = O(\frac{1}{R-1})$ providing the bulk of the sum), we thus have the pointwise bound
\begin{equation}\label{balk} \mathrm{Bal}_R(\zero) = \mathrm{Bal}_R(\deriv) + O\left(  \frac{\log \frac{1}{R-1}}{n(R-1)^2}  \right)
\end{equation}
for any $1 < R \leq 3/2$.

\subsection{Proof of (i)}  Now let us specialize to the case $a^{(\infty)}=0$, so that $a=o(1)$.  From Lemma \ref{basic}, $\deriv$ takes values in the thin lune $\overline{D(0,1)} \backslash \overline{D(a,1)}$, which lies in the $O(a)$-neighborhood of the semicircle $C$.  Taking limits, we conclude that $\deriv^{(\infty)}$ almost surely takes values in $C$.

Now we claim that $\zero^{(\infty)}$ is also supported on $C$.  By inner regularity it suffices to show that $\P( \zero^{(\infty)} \in K) = 0$ for all fixed compact subsets $K$ of $\overline{D(0,1)} \backslash C$.  The set $\{ \theta \in (-\pi/2,\pi/2): e^{i\theta} \in K \}$ is compact and thus can be contained in the interior of a fixed compact interval $I$ in $(-\pi/2,\pi/2)$ that depends only on $K$.  One can calculate\footnote{Indeed, the only difficult case is when $w$ is very close to the unit circle, but then one can show that $\int_{(-\pi/2,\pi/2] \backslash I} P^R_w(Re^{i\theta}) \frac{d\theta}{2\pi} \leq \frac{1}{2}$ (say), and the claim then follows since the Poisson kernel $P^R_w(Re^{i\theta})$ has mean one.} the lower bound
$$ \int_I P^R_w(Re^{i\theta}) \frac{d\theta}{2\pi} \gtrsim 1$$
uniformly for all $w \in K$ and $1 < R \leq 3/2$; probabilistically, this reflects the intuitive fact that a Brownian motion originating from $K$ will hit the circle $\partial D(0,R)$ first in the arc $A \coloneqq \{ Re^{i\theta}: \theta \in I\}$ with probability uniformly bounded from below; see Figure \ref{fig:harmonic}.  Note that our asymptotic notation allows the implied constant in the $\gtrsim$ notation to depend on fixed objects such as $K$ or $I$.

\begin{figure} [t]
\centering
\includegraphics[width=4in]{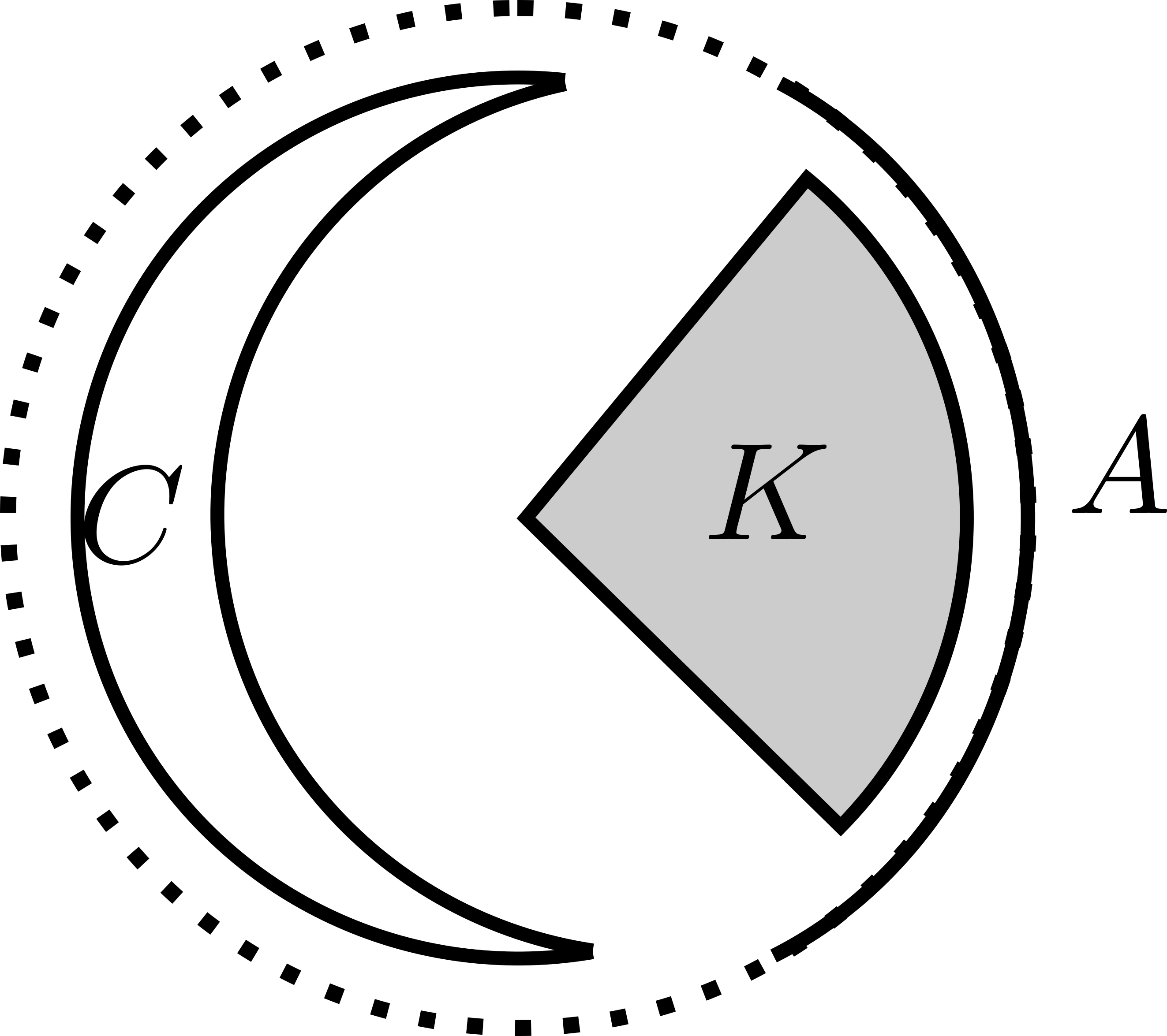}
\caption{A Brownian motion originating from the compact set $K$ is fairly likely to exit the disk $\overline{D(0,R)}$ (enclosed here by the dashed circle) in the arc $A$; but a Brownian motion originating from the semicircle $C$ (or from the thin lune between the two circular arcs on the left) is far more likely to exit in the complement of $A$ instead.  As a consequence, a pair $\zero, \deriv$ of random variables in $\overline{D(0,1)}$ cannot have the same (resp. almost the same) balayage if $\deriv$ is almost surely in (resp. near) $C$ and $\zero$ lies in $K$ with positive (resp. large) probability.}
\label{fig:harmonic}
\end{figure}

Setting $w = \zero^{(\infty)}$ and taking expectations, we conclude that
\begin{equation}\label{ibal}
 \int_I \mathrm{Bal}_R(\zero^{(\infty)}) \frac{d\theta}{2\pi} \gtrsim \P( \zero^{(\infty)} \in K).
\end{equation}
On the other hand, for any $w \in C$, one has
$$ \lim_{R \to 1^+} \int_I P^R_w(Re^{i\theta}) \frac{d\theta}{2\pi} = 0$$
uniformly in $w$; probabilistically, this reflects the intuitive fact that a Brownian motion originating from $C$ is very unlikely to hit the circle $\partial D(0,R)$ in the arc $A$.  Setting $w = \deriv^{(\infty)}$ and taking expectations, we conclude that
$$ \lim_{R \to 1^+} \int_I \mathrm{Bal}_R(\deriv^{(\infty)}) \frac{d\theta}{2\pi} = 0.$$
Applying \eqref{bal-bal}, \eqref{ibal} we conclude that $\P( \zero^{(\infty)} \in K)=0$ as desired.

As $\deriv^{(\infty)}, \zero^{(\infty)}$ almost surely take values in  $C \subset \partial D(0,1)$, the moments in \eqref{moment-match} can now be viewed as Fourier coefficients of the laws of $\deriv^{(\infty)}, \zero^{(\infty)}$. By Fourier uniqueness, we conclude that $\deriv^{(\infty)}, \zero^{(\infty)}$ have the same distribution as claimed.

Now we prove the more quantitative estimate \eqref{pzeta}, which uses a non-asymptotic version of the above arguments.  We may assume that $a$ is sufficiently small depending on $K$ since the claim is trivial otherwise.  Let $I$ be as before, and let $1 < R \leq 3/2$ be chosen later.  By repeating the proof of \eqref{ibal}, we see that
$$ \int_I \mathrm{Bal}_R(\zero) \frac{d\theta}{2\pi} \gtrsim \P( \zero \in K).$$
On the other hand, for $w$ in the lune $\overline{D(0,1)} \backslash \overline{D(a,1)}$, which lies in a $O(a)$-neighborhood of $C$, one can verify that\footnote{Indeed one has the uniform pointwise bound $P^R_w(Re^{i\theta}) \lesssim a + R-1$ for $\theta \in I$.}
$$ \int_I P^R_w(Re^{i\theta}) \frac{d\theta}{2\pi} \lesssim a + R-1$$
uniformly for $w$ in this lune, and hence
$$ \int_I \mathrm{Bal}_R(\deriv) \frac{d\theta}{2\pi} \lesssim a + R-1.$$
Applying \eqref{balk}, one concludes that
$$ \P( \zero \in K) \lesssim a + R-1 + \frac{\log \frac{1}{R-1}}{n(R-1)^2} .$$
Setting $R \coloneqq 1 + n^{-1/3}$, we obtain the claim.  This concludes the proof of Theorem \ref{limit-thm}(i).

\subsection{Proof of (ii)}  Now we suppose $a^{(\infty)}=1$, so that $a = 1 - o(1)$.  By passing to subsequences as necessary we may assume that $n$ is sufficiently large whenever this is required.

We will combine the above theory with the analysis of D\'egot \cite{degot}.  A key quantity will be the mean
$$ \mu \coloneqq \E \zero = \E \deriv$$
from Lemma \ref{basic}(ii).

First, from applying \cite[Theorem 5]{degot} (which relies ultimately on the powerful theorem of Grace \cite{grace}) we obtain the lower bound
\begin{equation}\label{labb}
 |f(\delta)| \geq \frac{1-\sqrt{1+\delta^2-\delta a}}{n} |f'(a)|
\end{equation}
for any $0 < \delta < a$, so in particular
\begin{equation}\label{fdel}
 |f(\delta)| \gtrsim \frac{1}{n} |f'(a)|
\end{equation}
whenever $1 \lesssim \delta \leq 1/2$.  Meanwhile, from \eqref{lfp} and the fact that $|a-\deriv| > 1$ we have
\begin{equation}\label{fan}
 |f'(a)| > n.
\end{equation}
On the other hand, by \cite[Lemma 1]{degot}, we have the elementary inequality
$$ |f(\delta)| \leq \left( \sqrt{1 + \delta^2 - 2 \delta \Re \mu} \right)^n.$$
Combining these three estimates, we conclude that for all $1 \lesssim \delta \leq 1/2$ we have
$$ \sqrt{1 + \delta^2 - 2 \delta \Re \mu} \geq 1 - O\left(\frac{1}{n}\right)$$
which implies in particular that
$$ \Re \mu \leq \frac{\delta}{2} + o(1).$$
Sending $\delta$ to zero arbitrarily slowly with $n$ (cf. Example \ref{underspill}), we conclude that
$$ \Re \mu \leq o(1)$$
(compare with \cite[Corollary 1, Remark 2]{degot}). On the other hand, since $\Re \deriv$ is bounded below by $-1$, we have
$$ \Re \mu \geq \eps - \P(\Re \deriv \leq \eps)$$
and thus
\begin{equation}\label{xio}
 \P( \Re \deriv \leq \eps ) \gtrsim 1 
\end{equation}
for $n$ large enough and any fixed $\eps>0$.

Next, from \eqref{lfp} one has for any $0 \leq \delta \leq 1/2$ that
$$ \log |f'(a)| - \log |f'(\delta)| = (n-1) \E \log \frac{|a-\deriv|}{|\delta-\deriv|}.$$
Since $\deriv$ takes values in the lune $\overline{D(0,1)} \backslash \overline{D(a,1)}$, which lies to the left of the perpendicular bisector of $\delta$ and $a$, we see that $\log \frac{|a-\deriv|}{|\delta-\deriv|} \geq 0$, with equality only when $\deriv$ has real part $\frac{a+\delta}{2} = \frac{1+\delta}{2} - o(1)$.  As a consequence, if $\Re \deriv \leq 0.1$ (say), we see that $\log \frac{|a-\deriv|}{|\delta-\deriv|} \gtrsim 1$. From \eqref{xio}, we conclude that 
$$ \log |f'(a)| - \log |f'(\delta)| \gtrsim n$$
and thus
$$ f'(\delta) \lesssim e^{-cn} |f'(a)|$$
for all $0 \leq \delta \leq 1/2$ and some absolute constant $c>0$.  A similar argument also gives
$$ f'(z) \lesssim e^{-cn} |f'(a)|$$
for $z \in D(0.5,0.01)$ (say). By the fundamental theorem of calculus, this implies that
$$ f(0) = f(\delta) + O\left( e^{-cn} |f'(a)| \right)$$
for all $0 \leq \delta \leq 1/2$, hence by \eqref{fdel} (for say $\delta=1/2$) we have
$$ |f(0)| \gtrsim \frac{1}{n} |f'(a)|.$$
From Lemma \ref{basic}(iii) we have $|f(0)| \leq 1$; comparing this with \eqref{fdel}, \eqref{fan} we see that
\begin{equation}\label{ffd}
 |f(0)| \sim |f(\delta)| \sim 1; \quad |f'(a)| \sim n
\end{equation}
for all $0 \leq \delta \leq 1/2$ (cf., \cite[Theorem 6]{degot}).  In particular, $f(0)$ is non-zero.  A similar fundamental theorem of calculus argument also shows
\begin{equation}\label{similar}
 |f(z)| \sim 1
\end{equation}
uniformly for $z \in D(0.5, 0.01)$. In particular, $f$ has no zeroes in this disk (cf. \cite[Theorem 4]{degot}).

From \eqref{ffd} and Lemma \ref{basic}(iii) we have
$$ \E \log \frac{1}{|\zero|} = \frac{1}{n} \log \frac{1}{|f(0)|} \lesssim \frac{1}{n}$$
(giving the first estimate of \eqref{zetas}),  hence on taking limits (using Fatou's lemma)
$$ \E \log \frac{1}{|\zero^{(\infty)}|} = 0.$$
Thus $\zero^{(\infty)}$ almost surely lies in the unit circle $\partial D(0,1)$.  In particular, the logarithmic potential $U_{\zero^{(\infty)}}$ is harmonic in $D(0,1)$.  On the other hand, from \eqref{similar} and Lemma \ref{basic}(iii) we have
$$ U_\zero(z) = O\left( \frac{1}{n} \right)$$
uniformly for $z \in D(0.5, 0.01)$, thus on taking limits we see that $U_{\zero^{(\infty)}}$ vanishes on $D(0.5, 0.01)$.  By unique continuation for harmonic functions, we thus see that $U_{\zero^{(\infty)}}$ vanishes on all of $D(0,1)$.  Using the expansion \eqref{fourier}
we conclude that all the non-zero Fourier coefficients of $\zero^{(\infty)}$ vanish, thus by Fourier uniqueness $\zero^{(\infty)}$ is uniformly distributed on the unit circle $\partial D(0,1)$.  A standard computation (using Jensen's formula, \eqref{Taylor}, or the two-dimensional version of the Newton shell theorem) then gives
\begin{equation}\label{umulog}
 U_{\zero^{(\infty)}}(z) = \min\left(\log \frac{1}{|z|}, 0\right)
\end{equation}
for $z \in \C$.

Returning to \eqref{ffd}, we use Lemma \ref{basic}(iii) again to conclude that
$$ \E \log |\deriv-a| = \frac{1}{n} \log \frac{|f'(a)|}{n} \lesssim \frac{1}{n}$$
(giving the second estimate of \eqref{zetas}) and hence on taking limits
$$ \E \log |\deriv^{(\infty)}-1| = 0.$$
As $\deriv^{(\infty)}$ was already almost surely in $\overline{D(0,1)} \backslash \overline{D(1,1)}$, it is thus almost surely in the arc $\overline{D(0,1)} \cap \partial D(1,1)$ (cf., \cite[Remark 3]{degot}).  In particular, the logarithmic potential $U_{\deriv^{(\infty)}}$ is harmonic outside of this arc. By \eqref{mun}, \eqref{umulog}, $U_{\deriv^{(\infty)}}$ agrees with $\log \frac{1}{|z|}$ in $\overline{D(0,1)}^c$; by unique continuation of harmonic functions, it thus agrees with $\log \frac{1}{|z|}$ almost everywhere in $\C$ (here we use the connectedness of the complement of the arc $\overline{D(0,1)} \cap \partial D(1,1)$ in the complex plane $\C$).  Taking distributional Laplacians using \eqref{mu-u}, we conclude that the law of $\deriv^{(\infty)}$ is the Dirac probability measure at $0$.  This proves Theorem \ref{limit-thm}(ii).

\begin{remark}\label{degot-time}  A similar argument applies when $0 < a^{(\infty)} < 1$ to conclude that $\deriv^{(\infty)}$ is almost surely $0$ and simultaneously $|\deriv^{(\infty)}-a^{(\infty)}|$ is almost surely $1$, which is absurd.  This gives a slightly different way to recover the main results of D\'egot \cite{degot} (although many of the key ingredients in both arguments are the same).
\end{remark}

\section{Getting a contradiction near the origin}\label{origin-sec}

In this section we treat the case where $a$ is somewhat close to the origin:

\begin{theorem}[Contradiction near the origin]\label{contra-origin}  Theorem \ref{main-contr} holds when $a = o(\frac{1}{\log n})$.  
\end{theorem}

We now prove this theorem.  The strategy is to control the number of zeroes of $f$ near the origin by applying the argument principle to the log derivative $\frac{f'}{f}$, with the aim to show that this number is either zero or negative, contradicting the hypothesis of having a zero $a$ near the origin.

By passing to a subsequence if necessary we may assume that $n$ is sufficiently large.

From Theorem \ref{limit-thm}(i) we know that in this case $\deriv^{(\infty)}, \zero^{(\infty)}$ are almost surely in the semicircle $C = \{ e^{i\theta}: \pi/2 \leq \theta \leq 3\pi/2\}$ with the same distribution, and that 
\begin{equation}\label{zeta-alone}
 \P( |\zero| \leq 1/2 ) = o\left(\frac{1}{\log n}\right)
\end{equation}
In other words, $\overline{D(0,1/2)}$ contains at most $o(n/\log n)$ of the zeroes of $f$ (counting multiplicity).

A key function in our analysis will be the Stieltjes transform $s_{\zero^{(\infty)}}$ of $\zero^{(\infty)}$.  This function is fixed and holomorphic on the connected domain $\C \backslash C$, and is not identically zero since it is asymptotic to $1/z$ as $|z| \to \infty$.  Hence the zeroes of this transform on $\C \backslash C$ are isolated.  In particular, one can find a fixed annulus $\overline{D(0,r_2)} \backslash D(0,r_1)$ for some fixed $0 < r_1 < r_2 < 1/2$ on which $s_{\zero^{(\infty)}}$ is bounded away from zero, thus
$$ |s_{\zero^{(\infty)}}(z)| \sim 1$$
uniformly for $r_1 \leq |z| \leq r_2$.

Let $m$ denote the number of zeroes of $s_{\zero^{(\infty)}}$ in the disk $D(0,r_1)$, then $m$ is a fixed non-negative integer, and by the argument principle we see that for any radius $r_1 < r < r_2$, the contour $\theta \mapsto s_{\zero^{(\infty)}}(re^{i\theta}), 0 \leq \theta \leq 2\pi$ stays at a distance $\sim 1$ from the origin and winds exactly $m$ times anti-clockwise around the origin.

From Lemma \ref{basic}(iv) we have
$$ s_\zero(z) = \frac{1}{n} \frac{f'(z)}{f(z)}.$$
Since $\zero$ converges in distribution to $\zero^{(\infty)}$ and ranges in the compact set $\overline{D(0,1)}$, we see from the local integrability of the Stieltjes kernel $\frac{1}{z}$ we see that $\frac{1}{n} \frac{f'(z)}{f(z)}$ converges locally in $L^1(\C)$ to 
$s_{\zero^{(\infty)}}(z)$. 

Suppose for the moment that this convergence was in fact locally uniform instead of merely being locally in $L^1(\C)$, so that
$$ \frac{1}{n} \frac{f'(z)}{f(z)} = s_{\zero^{(\infty)}}(z)+o(1)$$
uniformly for $z \in \overline{D(0,r_2)} \backslash D(0,r_1)$.  Then by Rouche's theorem and the argument principle, for $n$ sufficiently large and any $r_1 < r < r_2$, the contour 
\begin{equation}\label{con}
\theta \mapsto \frac{1}{n} \frac{f'}{f}(re^{i\theta}), 0 \leq \theta \leq 2\pi 
\end{equation}
would also wind $m$ times anti-clockwise around the origin, and the number of zeroes of $f'$ in $D(0,r)$ minus the number of zeroes in $f$ in $D(0,r)$ (counting multiplicity) would be equal to $m$.  But $f'$ has no zeroes in $\overline{D(a,1)} \supset D(0,1/2) \supset D(0,r)$, and $f$ contains at least one zero $a$ in $D(0,r)$, hence we have
\begin{equation}\label{m-absurd}
 m \leq -1
\end{equation}
which is absurd since $m$ is a non-negative integer.

Now we need to get around the technical issue that the convergence is a little bit worse than locally uniform. Here we will take advantage of the relatively small number of zeroes in or near the annulus $\{ r_1 \leq |z| \leq r_2\}$, as provided by \eqref{zeta-alone}.  We use Lemma \ref{basic}(iv) and split 
$$ \frac{1}{n} \frac{f'(z)}{f(z)} = \E \frac{1_{|\zero|>1/2}}{z-\zero} + \E \frac{1_{|\zero| \leq 1/2}}{z-\zero},$$
where we use $1_E$ to denote the indicator of an event $E$. Since $\zero$ converges in distribution to $\zero^{(\infty)}$, which is supported in $C$, we have from \eqref{convert} (approximating $1_{|\zero|>1/2}$ by continuous functions in the usual fashion) to conclude that
$$ \E \frac{1_{|\zero|>1/2}}{z-\zero} = \E \frac{1}{z-\zero^{(\infty)}} + o(1) = s_{\zero^{(\infty)}}(z)  + o(1)$$
uniformly for $z \in D(0,r_2)$.  If $|z|=r$, we have from the triangle inequality that
$$ \left|\E \frac{1_{|\zero| \leq 1/2}}{z-\zero}\right| \leq \E \frac{1_{|\zero| \leq 1/2}}{\left|r-|\zero|\right|}.$$
To get a good bound for the right-hand side we exploit the freedom to select $r$ via the probabilistic method (cf. the standard proof of the Hadamard factorisation theorem, see e.g., \cite[Corollary 5.4]{ss}).  Whenever $|\zero| \leq 1/2$, direct calculation gives
$$ \int_{r_1}^{r_2} \frac{dr}{\max\left( \left|r-|\zero|\right|, n^{-10} \right)} \lesssim \log n$$
hence by \eqref{zeta-alone} and the Fubini--Tonelli theorem we have
$$ \int_{r_1}^{r_2} \E \frac{1_{|\zero| \leq 1/2}}{\max\left( \left|r-|\zero|\right|, n^{-10} \right)}\ dr = o(1).$$
Hence by Markov's inequality we see that
$$ \E \frac{1_{|\zero| \leq 1/2}}{\max\left( \left|r-|\zero|\right|, n^{-10} \right)} = o(1)$$
for all $r \in [r_1,r_2]$ outside of a set of measure $o(1)$ (we do \emph{not} require this set to be fixed).  Also, since there are at most $n$ zeroes, we see that one surely has $\left|r-|\zero|\right| \geq n^{-10}$ for all $r \in [r_1,r_2]$ outside of a set of measure $O( n^{-10+1} ) = o(1)$.  We conclude that for $n$ large enough, we can find $r \in [r_1,r_2]$ (depending on $n$) such that
$$ \E \frac{1_{|\zero| \leq 1/2}}{\left|r- |\zero|\right|} = o(1).$$
Putting all this together, we conclude that for this particular choice of $r$, one has
$$ \frac{1}{n} \frac{f'}{f}(re^{i\theta}) = s_{\zero^{(\infty)}}(re^{i\theta}) + o(1)$$
uniformly in $\theta \in [0,2\pi]$.  By Rouche's theorem, we conclude (again for $n$ large enough) that for this choice of $r$, the contour \eqref{con} winds around the origin $m$ times, so we again obtain the absurd conclusion \eqref{m-absurd}.  This concludes the proof of Theorem \ref{contra-origin}.

\begin{remark}  It is instructive to compare this argument with the near-counterexample in Example \ref{ex-1}. Here $\zero^{(\infty)}, \deriv^{(\infty)}$ are uniformly distributed on unit circle $\partial D(0,1)$.  Unlike the situation with Theorem \ref{limit-thm}(i), the essential range of these random variables is now all of the circle $\partial D(0,1)$ rather than just the semicircle $C$; in particular, this support now disconnects the origin from infinity, and we can no longer prevent the Stieltjes transform $s_{\zero^{(\infty)}}$ from vanishing in the disk $D(0,1)$; indeed this transform does vanish by the residue theorem.  As a consequence, we no longer possess an annulus in which one can use the argument principle to prevent $f$ from acquiring a zero (or $f'/f$ acquiring a pole) near the origin. 
\end{remark}

\begin{remark}\label{low-cover} A variant of this argument can also cover the range $n^{-O(1)} \leq a \leq o(1)$ as follows.  From \eqref{labb}, \eqref{fan} one has
$$ 
|f(\delta)| > 1-\sqrt{1+\delta^2-\delta a} = \frac{\delta(a-\delta)}{1+\sqrt{1-\delta(a-\delta)}}$$
for any $0 < \delta < a$, so in particular
$$ \log \frac{1}{|f(\delta)|} \lesssim \log \frac{1}{\delta}$$
for $0 < \delta \leq a/2$.  Applying Lemma \ref{basic}(iii), we conclude that
$$ \E \log \frac{1}{|\zeta-\delta|} \lesssim \frac{1}{n} \log \frac{1}{\delta}.$$
On the other hand, $\log \frac{1}{|\zeta-\delta|}$ is bounded from below by $\log \frac{1}{1+\delta} = -O(\delta)$.  Applying Markov's inequality, we conclude for any fixed compact $K \subset D(0,1)$ that
$$ \P( \zeta \in K ) \lesssim \frac{1}{n} \log \frac{1}{\delta} + \delta$$
if $n$ is large enough.  Setting $\delta \coloneqq n^{-C}$ for a large fixed $C$, we conclude in particular that
$$ \P( \zeta \in K ) \lesssim \frac{\log n}{n}$$
and this estimate can be used as a replacement for \eqref{pzeta} in the above arguments, in particular by establishing \eqref{zeta-alone} in this regime.  
\end{remark}

\section{A motivational near-counterexample}\label{nearer}

In Example \ref{ex-1} a near-counterexample was provided to Sendov's conjecture in which $a$ was close to (and in fact equal to) $1$.  In this section we analyze (somewhat informally) a more general family of near-counterexamples of this form.  Unsurprisingly, we will eventually be able to show that no element of this family is actually a counterexample to the conjecture, but the task of definitively ruling out a counterexample turns out to be surprisingly delicate, and will serve to motivate the arguments in the next section.  However, none of the discussion in this section is strictly necessary for the proof of Theorem \ref{main}, and the impatient reader may skip ahead to the next section if desired.

Let $c_2 > c_1 > 0$ be fixed real constants, and let 
\begin{equation}\label{pj}
 P(z) = (z-\lambda_1) \dots (z-\lambda_m)
\end{equation}
be a fixed monic polynomial with distinct zeroes $\lambda_1,\dots,\lambda_m \in \C$, which we will choose later.  For $n > m$, we consider monic degree $n$ polynomials of the form
\begin{equation}\label{fad}
 f(z) \coloneqq \left(z + \frac{c_2}{n}\right)^{n-m} P(z) - \left(a + \frac{c_2}{n}\right)^{n-m} P(a) 
\end{equation}
where 
$$ a \coloneqq 1 - \frac{c_1}{n} = 1 - o(1).$$
Clearly $f$ has a zero at $a$.  The derivative can be computed as
\begin{equation}\label{fderiv}
f'(z) = (n-m) \left(z + \frac{c_2}{n}\right)^{n-m-1} \left( P(z) + \frac{z + \frac{c_2}{n}}{n-m} P'(z) \right)
\end{equation}
and thus has a zero of order $n-m-1$ at $-\frac{c_2}{n}$, which lies just outside of $\overline{D(a,1)}$ thanks to our hypothesis $c_2 > c_1$, while also inheriting the $m$ zeroes of $P(z) + \frac{z + \frac{c_2}{n}}{n-m} P'(z)$, which will take the form $\lambda_1+o(1),\dots,\lambda_m+o(1)$ by Rouche's theorem (indeed one can easily upgrade the $o(1)$ errors here to $O(1/n)$ if desired).  (In particular, a random zero $\deriv$ of $f'$ will converge to zero in probability, but not necessarily converge uniformly, which is consistent with Theorem \ref{limit-thm}(ii).)  If we now make the hypothesis that the zeroes $\lambda_1,\dots,\lambda_m$ lie in the open lune $D(0,1) \backslash \overline{D(1,1)}$ (so in particular $P(a) \neq 0$), then all the zeroes of $f'$ lie outside of $\overline{D(a,1)}$ when $n$ is large enough.

We would therefore obtain a contradiction to Sendov's conjecture if we could ensure that all zeroes $\zero$ of $f$ stayed inside the disk $\overline{D(0,1)}$ for large enough $n$.  
If $\zero$ is a zero of $f$, then from \eqref{fad} one has
$$ 
\left|\zero + \frac{c_2}{n}\right|^{n-m} |P(\zero)| = \left|a + \frac{c_2}{n}\right|^{n-m} |P(a)|.$$
The right-hand side is $\sim 1$, which (because the zeroes of $P$ are in the open disk $D(0,1)$) already forces
\begin{equation}\label{zon}
 \left|\zero + \frac{c_2}{n}\right| = 1 + O\left(\frac{1}{n}\right).
\end{equation}
By taking logarithms and using \eqref{pj}, we can rewrite the previous equation as
\begin{equation}\label{ten}
\log \left|\frac{\zero + \frac{c_2}{n}}{a + \frac{c_2}{n}}\right| = \frac{1}{n-m} \sum_{j=1}^m \log \left| \frac{a-\lambda_j}{\zero - \lambda_j}\right|.
\end{equation}
If we take an ansatz
$$ \zero = \left(1 + \frac{t}{n}\right) e^{i\theta} - \frac{c_2}{n}$$
for some real numbers $t,\theta$ with $t=O(1)$ (to be compatible with \eqref{zon}), we have the Taylor expansions
\begin{align*}
 \log \left|\frac{\zero + \frac{c_2}{n}}{a + \frac{c_2}{n}}\right| &= \frac{t}{n} - \frac{c_2-c_1}{n} + O\left( \frac{1}{n^2} \right)\\
 \log \left| \frac{a-\lambda_j}{z - \lambda_j}\right| &= \log \left|\frac{1 - \lambda_j}{e^{i\theta} - \lambda_j}\right| + O\left(\frac{1}{n}\right) \\
\frac{1}{n-m} &= \frac{1}{n} + O\left( \frac{1}{n^2} \right); 
\end{align*}
inserting these expansions into \eqref{ten} and then multiplying by $n$, we conclude that
$$ t = c_2 - c_1 + \sum_{j=1}^m \log \left|\frac{1 - \lambda_j}{e^{i\theta} - \lambda_j}\right| + O\left(\frac{1}{n}\right).$$
On the other hand, from a further Taylor expansion we have
$$ |\zero| = 1 + \frac{t - c_2 \cos \theta}{n} + O\left( \frac{1}{n^2} \right)$$
so for $n$ large enough we would be able to guarantee that all zeroes $\zero$ were in the disk $\overline{D(0,1)}$ if we could ensure the inequality
\begin{equation}\label{lamin}
c_2 - c_1  - c_2 \cos \theta + \sum_{j=1}^m \log \left|\frac{1 - \lambda_j}{e^{i\theta} - \lambda_j}\right|  < 0
\end{equation}
for all $0 \leq \theta \leq 2\pi$.  

One can rule out this inequality by taking an average in $\theta$ (or equivalently, extracting the zeroth Fourier coefficient).  Indeed, observe from Jensen's formula (or \eqref{Taylor}) that $\log \left|\frac{1 - \lambda_j}{e^{i\theta} - \lambda_j}\right|$ has a mean value of $\log |1-\lambda_j|$ for $\theta \in [0,2\pi]$; meanwhile, $c_2 \cos \theta$ has a mean of zero.  Thus if \eqref{lamin} held for all $0 \leq \theta \leq 2\pi$, then we would have
$$ c_2 - c_1 + \sum_{j=1}^m \log |1-\lambda_j| < 0.$$
But from hypothesis the quantities $c_2 - c_1$ and $\log |1-\lambda_j|$ are non-negative, which (barely) gives a contradiction.

This rules out the most naive effort to generate counterexamples to Sendov's conjecture using this ansatz.  But one might hope that by some more careful analysis of the various $O(1/n^2)$ errors that were discarded that one might be able to still salvage a counterexample assuming only the non-strict inequalities
\begin{equation}\label{dd}
c_2 - c_1 - c_2 \cos \theta + \sum_{j=1}^m \log \left|\frac{1 - \lambda_j}{e^{i\theta} - \lambda_j}\right| \leq 0
\end{equation}
and similarly relaxing the conditions $c_2 > c_1$, $\lambda_j \in D(0,1) \backslash \overline{D(1,1)}$ to their non-strict counterparts $c_2 \geq c_1$, $\lambda_j \in \overline{D(0,1)} \backslash D(1,1)$.  Then the same averaging argument as before will now give
$$ c_2 - c_1 + \sum_{j=1}^m \log |1-\lambda_j| \leq 0$$
and this forces $c_2-c_1$ and $\log |1-\lambda_j|$ to vanish.  Thus the only surviving endpoint is when $c_1=c_2$ and when all the $\lambda_j$ lie on the arc $\overline{D(0,1)} \cap \partial D(1,1)$.  The left-hand side of \eqref{dd} now is non-positive while also having zero mean, and must be identically zero, thus
\begin{equation}\label{ido}
- c_2 \cos \theta + \sum_{j=1}^m \log \left|\frac{1 - \lambda_j}{e^{i\theta} - \lambda_j}\right| = 0
\end{equation}
for all $\theta \in [0,2\pi]$.  

We have exhausted the information we can glean from the zeroth Fourier coefficient.  The first Fourier coefficient ends up providing the constraint $\sum_{j=1}^m \lambda_j = c_2$, which is interesting (it reflects the tendency of the mean $\mu = \E \zero = \E \deriv$ to be surprisingly close to zero, which we will take advantage of in the next section) but does not end up being directly useful for the purpose of establishing a contradiction. To get such a contradiction we will instead take the \emph{second} Fourier coefficient, which allows us to ignore the cosine term in \eqref{ido}.  Indeed, from \eqref{Taylor} and Fourier analysis we have
$$ \int_0^{2\pi} \log \left|\frac{1 - \lambda_j}{e^{i\theta} - \lambda_j}\right| e^{-2i \theta} \frac{d\theta}{2\pi} = \frac{\lambda_j^2}{4}$$
and
$$\int_0^{2\pi} \cos \theta e^{-2i \theta} \frac{d\theta}{2\pi} = 0.$$
Thus, from extracting the second Fourier coefficient of \eqref{ido} we obtain
\begin{equation}\label{summ}
 \sum_{j=1}^m \lambda_j^2 = 0.
\end{equation}
Now we recall the key geometric observation from Figure \ref{fig:key}: the zeroes $\lambda_j$ lie on the arc $\overline{D(0,1)} \backslash D(1,1)$ and thus either vanish, or have argument either in $[\pi/3,\pi/2]$ or $[-\pi/2,-\pi/3]$.  Thus the squares $\lambda_j^2$ either vanish, or have argument in $[2\pi/3, 4\pi/3]$; in particular, in the latter case the real part is negative (with some ``room'' to spare\footnote{In the earlier analysis of Miller \cite{miller} and Chijiwa \cite{chijiwa-disk}, \cite{chijiwa}, \cite{kasmalkar} there is even more room (remarked upon as ``quite startling'' in the discussion after \cite[Theorem 3]{miller}), since in this regime the zeroes $\deriv$ are so close to the origin that the argument is now very close to $\pm \pi/2$.  However, in our situation $\deriv$ will only converge to zero in distribution, and as the example \eqref{fad} illustrates, one can have a few zeroes of $f'$ at ``macroscopic'' distances from the origin that also need to be carefully treated.} since $[2\pi/3,4\pi/3]$ is strictly contained in $(\pi/2,3\pi/2)$).  Thus we see that the inequality \eqref{summ} can only hold for all $\theta \in [0,2\pi]$ in the degenerate case when $c_2=c_1$ and all the $\lambda_j$ vanish, which also forces the remaining term $c_2 \cos \theta$ to also vanish, and we collapse back to the near-counterexample in Example \ref{ex-1}.  Thus this attempt to generate a counterexample to Sendov's conjecture has definitively failed.  However, the reason for the failure was rather subtle, and required one to perform Taylor expansions to accuracy $O(1/n^2)$ (actually an accuracy of $o(1/n)$ would have sufficed for us), as well as an analysis of the possible solutions to the system of inequalities provided by \eqref{dd}.  Ultimately one had to utilise the geometric property about the argument of points on the arc $\overline{D(0,1)} \backslash D(1,1)$ depicted in Figure \ref{fig:key}.  On the other hand, there was some room in the argument at the end, which gives hope that one could adapt this argument to more general polynomials $f$, as long as one had a sufficient level of accuracy in one's Taylor expansions.  This belief is also buttressed by the analysis in previous work in the $a^{(\infty)}=1$ regime in \cite{miller}, \cite{vz}, \cite{chijiwa}, \cite{chijiwa-disk}, \cite{kasmalkar}, which also extracts a contradiction by performing Taylor expansions to comparable levels of accuracy to the ones employed here.  This serves as motivation for the strategy used to rigorously prove the remaining case needed for Theorem \ref{main-contr}, to which we now turn.

\section{Getting a contradiction near the unit circle}\label{circle-sec}

In view of Theorem \ref{contra-origin} and the discussion in the introduction, to conclude the proof of Theorem \ref{main} it will suffice to establish

\begin{theorem}[Contradiction near the unit circle]\label{contra-circle}  Theorem \ref{main-contr} holds when there is a fixed $\eps_0>0$ such that
\begin{equation}\label{hypo}
 1-o(1) \leq a \leq 1 - \eps_0^n
\end{equation}
(so in particular $a^{(\infty)} = 1$).  
\end{theorem}

As discussed in the previous section, the analysis required to prove Theorem \ref{contra-circle} is more delicate than that required to prove Theorem \ref{contra-origin}, due to the need to eliminate near-counterexamples such as \eqref{fad}.  In the case when $a$ was extremely close to $1$, this was achieved in \cite{miller}, \cite{vz} by Taylor expanding out all relevant expressions to second order in the zeroes $\deriv$ of $f'$ (with third-order terms $O( |\deriv|^3 )$ being ultimately negligible).  We will adopt a similar strategy here, in particular finding approximations to the zeroes of $f$ that are accurate to an error of $o( \sigma^2 ) + o(1)^n$, where $\sigma^2$ is the variance of $\deriv$ (cf. \cite[Lemma 7]{miller}).  The presence of the superexponentially small error $o(1)^n$ is the reason why we need a weak upper bound on $a$ in \eqref{hypo}.

We now begin the proof.  It will be useful to work with the mean
$$ \mu \coloneqq \E \deriv = \E \zero$$
and variance
$$ \sigma^2 \coloneqq \E |\deriv - \mu|^2 = \E |\deriv|^2 - |\mu|^2$$
of $\deriv$.  Since $\deriv$ takes values in $\overline{D(0,1)}$ and converges in probability to zero thanks to Theorem \ref{limit-thm}(ii), we see from \eqref{convert} that these quantities are small:
\begin{equation}\label{xi-mom}
\mu, \sigma = o(1).
\end{equation}
Also we clearly have the identity
\begin{equation}\label{var-ident}
\E |\deriv|^2 = |\mu|^2 + \sigma^2.
\end{equation}
In practice, the $\sigma^2$ term in \eqref{var-ident} will dominate the $|\mu|^2$ term; see \eqref{meao} below.

The convergence in probability of $\deriv$ to zero is also enough to establish some preliminary control on $f$ and its derivatives.  

\begin{lemma}[Preliminary bounds]\label{basic-bounds}  For any fixed compact subset $K \subset \C$, we have the bounds
$$ f(z) = f(0) + O( (|z|+o(1))^n )$$
and
$$ f'(z) = O( (|z|+o(1))^n )$$
for all $z \in K$.
\end{lemma}

These bounds may be compared with the situation in Example \ref{ex-1} or in \eqref{fad}, \eqref{fderiv}.

\begin{proof}  We may take $K = \overline{D(0,R)}$ for a fixed $R>0$. By the underspill principle (Example \ref{underspill}), it suffices to show that for any fixed $0 < \eps < R$, one has
$$ |f(z) - f(0)| \leq(|z|+\eps+o(1))^n$$
and
$$ |f'(z)| \leq (|z|+\eps+o(1))^n$$
for all $z \in \overline{D(0,R)}$.  The first bound is implied by the second by the fundamental theorem of calculus, so it suffices to prove the latter bound.

For all $z \in \overline{D(0,R)} \backslash D(0,\eps)$, one has from \eqref{convert} that
\begin{align*}
 \E \log |z-\deriv| &\leq \E \log \max( |z-\deriv|, \eps/2)\\
&= \E \log \max( |z-\deriv^{(\infty)}|, \eps/2) + o(1)\\
&= \log \max(|z|,\eps/2) + o(1)\\
&= \log |z| + o(1).
\end{align*}
In particular, from Lemma \ref{basic}(iii) we have the bound
$$ |f'(z)| = n \exp( (n-1) \E \log |z-\deriv| ) \leq e^{o(n)} |z|^{n-1}$$
for all $z \in \overline{D(0,R)} \backslash D(0,\eps)$ (note that the factor of $n$ can be absorbed into the $e^{o(n)}$ term).  By the maximum principle, we conclude that
$$ |f'(z)| \leq e^{o(n)} \max(|z|,\eps)^{n-1}$$
for all $z \in \overline{D(0,R)}$, and the claim follows.
\end{proof}

It would simplify the arguments below significantly if we could also make the stronger assertion that $\deriv$ converges \emph{uniformly} to $0$.  This is unfortunately not easy to do (basically because of the presence of near-counterexamples such as \eqref{fad}, whose derivative $f'$ does possess zeroes $\deriv$ that converge to non-zero values), but we have the following weaker substitute:

\begin{proposition}[Uniform convergence of $\deriv$]\label{conv}  After passing to a subsequence of $n$ if necessary, there exists a fixed compact subset $S \subset \overline{D(0,1)}$ of the form
\begin{equation}\label{S-def}
 S = (\overline{D(0,1)} \cap \partial D(1,1)) \cup T
\end{equation}
where $T$ is a fixed and at most countable subset of the lune $\overline{D(0,1)} \backslash \overline{D(1,1)}$, all elements of which are isolated points of $S$, with the property that $\mathrm{dist}(\deriv, S)$ converges uniformly to zero.  Equivalently (by Example \ref{unif}), there exists $\eps = o(1)$ such that $\mathrm{dist}(\deriv,S) \leq \eps$ surely.
\end{proposition}

We remark that in the example \eqref{fad}, the set $T$ can be taken to be the zero set $\{\lambda_1,\dots,\lambda_m\}$ of the polynomial $P$.

\begin{proof} From Lemma \ref{basic}(i) we see $\deriv$ surely lies in a lune of the form $\overline{D(0,1)} \backslash D(1,1-o(1))$.  Now let $m$ be a fixed natural number.  For $n$ large enough, we see from Theorem \ref{limit-thm}(ii) and Markov's inequality that
$$ \P\left( \log |\deriv - a| \geq \frac{1}{2m} \right) \lesssim 1/n $$
and hence for $n$ large enough there are only $O(1)$ zeroes of $f$ in the partial annulus
\begin{equation}\label{annular}
\overline{D(0,1)} \cap \overline{D\left(1,1+\frac{1}{m}\right)} \backslash D\left(1,1+\frac{1}{m+1}\right)
\end{equation}
(not counting multiplicity).  Passing to a subsequence of $n$, we may assume (by the infinite pigeonhole principle) that the number $k_m$ of such zeroes is fixed, and then by the Bolzano--Weierstrass theorem we can then find fixed complex numbers $z_{m,1},\dots,z_{m,k_m}$ in the compact set \eqref{annular}, with the property that any zero $\zero$ of $f$ that lies in \eqref{annular} is of the form $z_{m,j}+o(1)$ for some $j=1,\dots,k_m$.  By repeatedly passing to a subsequence in this fashion for $m=1,2,\dots$, and then passing to a diagonal subsequence (as in the proof of the Arzel\'a--Ascoli theorem), we can then reach a subsequence for which there exist fixed complex numbers $z_{m,1},\dots,z_{m,k_m}$ in \eqref{annular} for \emph{each} fixed natural number $m$, with the property that any zero $\zero$ of $f$ that lies in \eqref{annular} for one of these $m$ is of the form $z_{m,j}+o(1)$ for some $j=1,\dots,k_m$.  If we then set $T \coloneqq \bigcup_{m=1}^\infty \{z_{m,1},\dots,z_{m,k_m}\}$ and let $S$ be the set \eqref{S-def}, we see that $\zero$ surely lies within $o(1)$ of $T$, and the claim follows.
\end{proof}

We remark that Theorem \ref{limit-thm}(ii) and some additional effort in fact yields the bound $\sum_{z \in T} \log |z-1| < \infty$, but we will not need this bound in our arguments.

Henceforth we pass to a subsequence and extract the sets $S,T$ provided by Proposition \ref{conv}.  When one is far from $S$, one is able to obtain good approximations for $f$ and $f'$:

\begin{proposition}[Approximating $f, f'$ outside of $S$]\label{f-approx}  Let $K$ be a fixed compact subset of $\C \backslash S$.  
\begin{itemize}
\item[(i)]  (Approximation of $f'$) One has
$$ f'(w) = \left(1 + O\left( n |z-w| \sigma^2 e^{o(n |z-w|)}\right)\right) \frac{f'(z)}{(z - \mu)^{n-1}} (w - \mu)^{n-1} $$
for all $z,w \in K$.
\item[(ii)]  (Approximation of $f$) For any fixed $\eps>0$, one has
\begin{equation}\label{fz-f0}
 f(z) = f(0) + \frac{1 + O( \sigma^2 )}{n} f'(z) (z - \mu) + O( (\eps+o(1))^n )
\end{equation}
for all $z \in K$.
\end{itemize}
\end{proposition}

Again, one is encouraged to compare these bounds against Example \ref{ex-1} or \eqref{fad}.  (In the latter case, the random variable $\deriv$ is equal to $\frac{-c_2}{n}$ with probability $1-\frac{m}{n-1}$, and is close to a randomly chosen zero $\lambda_j$ of $P$ otherwise.  In particular, we have $\mu, \sigma^2 = O(1/n)$ in this case.)  The reader may find it helpful to think of the special case in which the compact set $K$ lies on ``one side'' of the arc $\overline{D(0,1)} \cap \partial D(1,1)$, for instance in the lune $D(0,1) \backslash D(1,1)$ or in the complementary set $D(0,1) \cap D(1,1)$, although the estimates here also happen to be valid in ``disconnected'' cases in which the compact set $K$has components on either side of the arc.

\begin{proof}  Note that we may assume that $n$ is sufficiently large depending on $K$, since the claim is trivial otherwise.
Assuming this, we see from \eqref{xi-mom} and Proposition \ref{conv} we see that $|z-\deriv|, |z-\E \deriv| \sim 1$ for all $z \in K$, so in particular the function $\frac{f''}{f'}$ is holomorphic on $K$.  Also we have the Taylor expansion
$$ \frac{1}{z-\deriv} = \frac{1}{z-\mu} + \frac{\deriv-\mu}{(z-\mu)^2} + O( |\deriv-\mu|^2 )$$
surely and uniformly for $z \in K$, which on taking expectations gives the approximation 
$$ s_\deriv(z) = \frac{1}{z-\mu} + O( \sigma^2 )
$$
uniformly for $z \in K$.  Applying Lemma \ref{basic}(vi), we conclude that for any contour $\gamma: [0,1] \to K$ of length $|\gamma|$, that
\begin{equation}\label{gammas}
 f'(\gamma(1)) = e^{O( n |\gamma| \sigma^2 )} \frac{(\gamma(1) - \mu)^{n-1}}{(\gamma(0) - \mu)^{n-1}} f'(\gamma(0)).
\end{equation}

The set $\C \backslash S$ is open and path-connected. By compactness one can easily verify that for any fixed compact set $K$ of $\C \backslash S$, there exists a larger fixed compact subset $K'$ of $\C \backslash S$ such that any pair of points $z,w$ in $K$ are connected by a contour $\gamma$ in $K'$ of length $O(|z-w|)$.  Applying \eqref{gammas} (with $K$ replaced by $K'$) we conclude that
$$ f'(w) = e^{O\left( n |z-w| \sigma^2  \right)} \frac{(w - \mu)^{n-1}}{(z - \mu)^{n-1}} f'(z).
$$
for $w,z \in K$.  From the fundamental theorem of calculus we have the estimate $e^z = 1 + O( |z| e^{|z|} )$ uniformly for $z \in \C$, thus
$$ f'(w) = \left(1 + O\left( n |z-w| \sigma^2 e^{  O( n |z-w| \sigma^2 ) } \right)\right) \frac{(w - \mu)^{n-1}}{(z - \mu)^{n-1}} f'(z).
$$
Applying \eqref{xi-mom} we obtain part (i).  Integrating, we see that 
\begin{equation}\label{silence}
\begin{split}
&f(\gamma(0)) = f(\gamma(1)) - \frac{f'(\gamma(0))}{(\gamma(0) - \mu)^{n-1}} \times \\
&\quad \left(  \int_\gamma (w - \mu)^{n-1}\ dw + O\left( n \sigma^2 \int_\gamma |\gamma(0)-w| e^{o(n |\gamma(0)-w|)} |w-\mu|^{n-1} \ |dw| \right)\right)
\end{split}
\end{equation}
for any contour $\gamma: [0,1] \to K$ in a fixed compact subset $K$ of $\C \backslash S$.

Now let $K$ a fixed compact subset of $\C \backslash S$, and let $\eps > 0$ be fixed.  For the purpose of proving (ii) we may assume that $\eps$ is sufficiently small depending on $K$.  Observe that for every $z \in K$, there exists a smooth contour $\gamma: [0,1] \in \C \backslash S$ with $\gamma(0) = z$, $|\gamma(1)| = \eps$, which is always moving towards the origin in the sense that one has the estimates
\begin{equation}\label{gamm}
 \gamma'(t) = O(1); \quad - z \cdot \gamma'(t) \gtrsim 1
\end{equation}
for all $0 \leq t \leq 1$, where we use the inner product $z \cdot w \coloneqq \Re z \overline{w}$, and the implied constants here are allowed to depend on $\eps$.  This is geometrically obvious if we allow $\gamma$ to lie in the complement $\C \backslash (\overline{D(0,1)} \cap \partial D(1,1))$; when one adds in the set $T$ there are a finite number of points that the previously constructed curve $\gamma$ may pass through, but one can easily avoid those points by a slight perturbation of this curve. In fact one can take $\gamma$ to be a circular arc if desired: see Figure \ref{fig:curve}.

\begin{figure} [t]
\centering
\includegraphics[width=4in]{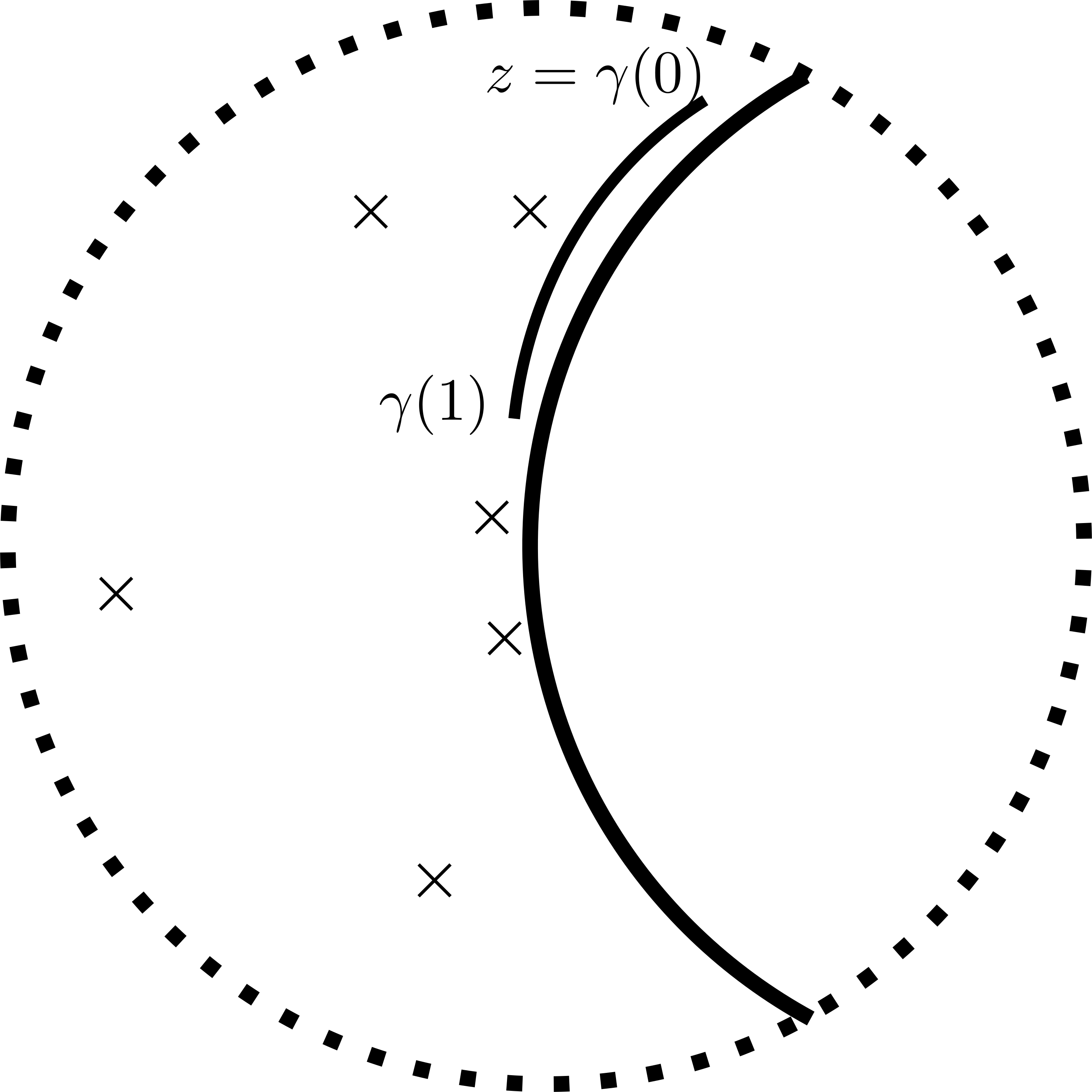}
\caption{In this picture the dashed circle is $\partial D(0,1)$, and the compact set $S$ consists of the circular arc $\overline{D(0,1)} \cap \partial D(1,1)$ and the six points $T$ marked with crosses.  Given any fixed point $z$ in the complement of $S$ and given any fixed $\eps$ with $0 < \eps < |z|$, one can find a contour $\gamma$ avoiding $S$ that starts at $z$ and ends at a point of magnitude $\eps$, which is always moving towards the origin in the sense of \eqref{gamm}. In the absence of the points $T$, one can easily achieve this by choosing $\gamma$ to be an arc on a suitable circle connecting $z$ and the origin (in most cases one can take the circle whose tangent at the origin is vertical); if this arc happens to pass through some of the points of $T$, one can easily perturb the arc to avoid them since these elements are isolated in $S$.}
\label{fig:curve}
\end{figure}

From the compactness of $K$ one can ensure that the bounds \eqref{gamm} hold uniformly for all $z \in K$, and that $\gamma$ ranges in a fixed compact set $K'$ of $\C \backslash S$.  From the fundamental theorem of calculus and \eqref{xi-mom} we have
$$ \int_\gamma (w - \mu)^{n-1}\ dw = \frac{(z - \mu)^n}{n}  - O\left( (\eps+o(1))^n \right)$$
and hence by \eqref{silence} and Lemma \ref{basic-bounds} we have
\begin{align*}
 f(z) &= f(0) + O\left((\eps+o(1))^n\right) + \frac{1}{n} f'(z) (z - \mu) \\
&\quad + O\left( n |f'(z)| \sigma^2 \int_\gamma |z-w| \left(\frac{|w-\mu|}{|z - \mu|}\right)^{n-1} e^{o(n |z-w|)}\ |dw| \right).
\end{align*}
From \eqref{gamm} and the fundamental theorem of calculus, if $w = \gamma(t)$ lies on the curve $\gamma$, then $|z-w| \sim t$ and the inner product of $z-w$ with $z$ is $\sim t$, hence also (by \eqref{xi-mom})
$$ |w-\mu| \leq |z-\mu| - ct$$
for some fixed $c>0$, thus also
$$ \left(\frac{|w-\mu|}{|z - \mu|}\right)^{n-1} \lesssim e^{-cnt}.$$
We thus have
\begin{align*}
 f(z) &= f(0) + O\left((\eps+o(1)\right)^n + \frac{f'(z) (z - \mu)}{n} \\
&\quad + O\left( n |f'(z)| \sigma^2 \int_0^1 t e^{-(c-o(1))nt}\ dt \right)
\end{align*}
Estimating $c-o(1)$ from below by $c/2$ (for $n$ large enough) and evaluating the integral, we see that
$$ \int_0^1 t e^{-(c-o(1))nt}\ dt \lesssim n^{-2}$$
and thus we obtain the desired claim (ii) (note that $|z-\mu| \sim 1$ for $z \in K$ thanks to \eqref{xi-mom} and the fact that $0 \not \in K$).
\end{proof}

This gives us reasonably fine control on the location of the zeroes of $f$ away from $S$:

\begin{corollary}[Zeroes of $f$ away from $S$]
Let $K$ be a fixed compact subset of $\C \backslash S$, and let $\eps>0$ be fixed.  If $n$ is sufficiently large depending on $\eps,K$, then for any zero $\zero$ of $f$ inside $K$, one has
\begin{equation}\label{fine}
-U_\deriv(\zero) = \frac{1}{n} \log \frac{1}{|f(0)|} + O\left(\frac{\sigma^2}{n}\right) + O\left( (\eps+o(1))^n \right).
\end{equation}
\end{corollary}

The reader is encouraged to compare this equation to \eqref{ten}. With more effort one can perform a further asymptotic expansion of the $O(\sigma^2/n)$ error, but we will not need to do so here (in fact our analysis only requires the weaker bound of $o(\sigma^2)$ on this error).  Informally, this corollary asserts that the zeroes of $f$ lie very close to a level set of the logarithmic potential $U_\deriv$.

\begin{proof}  Applying \eqref{f-approx} one has
$$ 0 = f(0) + \frac{1 + O(\sigma^2)}{n} f'(\zero) (\zero - \mu) + O\left( (\eps+o(1))^n \right)$$
and thus by \eqref{ffd} and the triangle inequality
$$ \frac{|f'(\zero)|}{n} |\zero - \mu| = |f(0)| \left( 1 + O(\sigma^2) + O\left( (\eps+o(1))^n \right) \right).$$
Taking logarithms using Lemma \ref{basic}(iii), and dividing by $n$, we conclude that
\begin{equation}\label{non}
 - \frac{n-1}{n} U_\deriv(\zero) + \frac{1}{n} \log |\zero - \mu| = \frac{1}{n} \log |f(0)| + O\left(\frac{\sigma^2}{n}\right) + O\left( (\eps+o(1))^n \right).
\end{equation}
From Proposition \ref{conv} and \eqref{xi-mom} we have $|\zero - \deriv|, |\zero - \mu| \sim 1$ surely if $n$ is large enough.
From Taylor expansion we then have
$$ \log |\zero - \deriv| = \log |\zero - \mu| - \Re \frac{\deriv - \mu}{\zero - \mu} + O(\sigma^2)$$
and thus on taking expectations
$$ - U_\deriv(\zero) = \log |\zero - \mu| + O( \sigma^2 ).$$
If we use this to replace the $\log |\zero - \mu|$ term in \eqref{non} with $-U_\deriv(\zero) - O( \sigma^2 )$, we obtain the claim.
\end{proof}

The zero $a$ of $f$ can certainly be enclosed in a fixed compact subset $K$ of $\C \backslash S$ for $n$ large enough (e.g., in $\overline{D(1,0.1)}$), hence in particular we have from the above corollary
$$ -U_\deriv(a) = \frac{1}{n} \log \frac{1}{|f(0)|} + O\left(\frac{\sigma^2}{n}\right) + O\left( (\eps+o(1))^n \right).$$
By the triangle inequality, we conclude that
\begin{equation}\label{exi}
 U_\deriv(\zero) = U_\deriv(a) + O\left(\frac{\sigma^2}{n}\right) + O\left( (\eps+o(1))^n \right)
\end{equation}
for any zero $\zero$ of $f$ in a fixed compact subset $K$ of $\C \backslash S$, if $n$ is large enough depending on $K,\eps$.  We can combine this estimate with the requirement $\zero \in \overline{D(0,1)}$ to get further information on the zeroes $\deriv$ of $f'$:

\begin{proposition}[Fine control on $\deriv$]\label{xi-fine} Let $\eps>0$ be fixed.
\begin{itemize}
\item[(i)]  We have 
\begin{equation}\label{meao}
\mu \lesssim \sigma^2 + (\eps+o(1))^n
\end{equation}
(cf. \cite[Proposition 6(1)]{miller}) and
\begin{equation}\label{1a}
1-a \lesssim \sigma^2 + (\eps+o(1))^n.
\end{equation}
\item[(ii)]  If $I$ is a compact subset of $\partial D(0,1) \backslash S$, then
$$ U_\deriv(a) - U_\deriv(e^{i\theta}) \geq - o(\sigma^2) - o(1)^n$$
for all $e^{i\theta} \in I$, if $n$ is sufficiently large depending on $\eps, I$.  (Compare with \eqref{dd}.)
\item[(iii)]  For any fixed smooth function $\varphi \colon \partial D(0,1) \to \C$, we have
\begin{equation}\label{bbar}
 \int_0^{2\pi} \varphi(e^{i\theta}) U_\deriv(e^{i\theta})\ \frac{d\theta}{2\pi}
= o(\sigma^2) + o(1)^n 
\end{equation}
and also
\begin{equation}\label{abar}
 U_\deriv(a) = o(\sigma^2) + o(1)^n.
\end{equation}
\end{itemize}
\end{proposition}

\begin{proof}  We first prove (i).  Suppose that $K$ is a fixed compact subset of $\C \backslash S$.
If $\zero$ is a zero of $f$ in $K$, and $n$ is sufficiently large depending on $\eps,K$, then we can perform a Taylor expansion
$$ \log \frac{1}{|\zero - \deriv|} = \log \frac{1}{|\zero|} + \Re \frac{\deriv}{\zero} + O( |\deriv|^2 )$$
and
$$ \log \frac{1}{|a - \deriv|} = \log \frac{1}{|a|} + \Re \frac{\deriv}{a} + O( |\deriv|^2 )$$
and thus on taking expectations in $\deriv$ and applying \eqref{exi}, \eqref{var-ident} we obtain the crude estimate
$$ \log \frac{1}{|\zero|} + \Re \left(\left(\frac{1}{\zero} - \frac{1}{a}\right)\mu\right) = \log\frac{1}{|a|} +  O( |\mu|^2 + \sigma^2 ) + O\left( (\eps+o(1))^n \right).$$
Since the zero $\zero$ must lie in $\overline{D(0,1)}$, we conclude in particular that
\begin{equation}\label{xa}
 \Re\left(\left(\frac{1}{a} - \frac{1}{\zero}\right) \mu \right) \geq \log|a| +  O( |\mu|^2 + \sigma^2 ) + O\left( (\eps+o(1))^n \right).
\end{equation}

If $e^{i\theta}$ is an fixed element of $\partial D(0,1)$ that avoids $S$, then by Theorem \ref{limit-thm}(ii) there exists a zero $\zero$ of $f$ that lies within $o(1)$ of $e^{i\theta}$.  Applying \eqref{xa}, we conclude that
\begin{equation}\label{xb}
 \Re\left( \left(1 - e^{-i\theta} + o(1)\right) \mu \right) \geq \log|a| +  O( |\mu|^2 + \sigma^2 ) + O\left( (\eps+o(1))^n \right)
\end{equation}
if $n$ is sufficiently large depending on $\eps,\theta$.

Let $\theta_+ \in [0.98 \pi, 0.99 \pi]$, $\theta_- \in [1.01 \pi, 1.02 \pi]$ be fixed numbers such that $e^{i\theta_+}, e^{i\theta_-} \not \in S$.  For $n$ large enough depending on $\theta_+, \theta_-$, there will be a convex combination of the quantities $1 - e^{-i\theta_+} + o(1)$, $1 - e^{-i\theta_-} + o(1)$ appearing in \eqref{xb} that lies in the real interval $[1.9, 2]$.    We conclude that
$$ \Re \mu \geq \frac{\log|a|}{1.9} +  O( |\mu|^2 + \sigma^2 ) + O\left( (\eps+o(1))^n \right);$$
since $\log |a| = - (1-a) (1+o(1))$, we conclude that
\begin{equation}\label{rexi}
 \Re \mu \geq - \frac{1-a}{1.9 + o(1)} +  O( |\mu|^2 + \sigma^2 ) + O( (\eps+o(1))^n ).
\end{equation}
On the other hand, we have $|\deriv - a| \geq 1$ surely, hence on squaring
$$ |\deriv|^2 - 2 a \Re \deriv + a^2 \geq 1$$
hence on taking expectations and rearranging using \eqref{var-ident} and the fact that $\frac{1-a^2}{2a} = (1-a)(1+o(1))$, we obtain
$$ \Re \mu \leq -(1-a)(1+o(1)) + O( |\mu|^2 + \sigma^2 ).$$
Comparing this with \eqref{rexi} we conclude 
$$ \Re \mu, 1-a \lesssim |\mu|^2 + \sigma^2 + (\eps+o(1))^n$$
for $n$ large enough, and hence for all $n$ (since the $o(1)$ term on the right-hand side can be arbitrarily large for bounded $n$).
Inserting these bounds back into \eqref{xb}, we now conclude that
$$ 
(\sin \theta + o(1)) \Im \mu \leq O( |\mu|^2 + \sigma^2 ) + O( (\eps+o(1))^n )$$
whenever $e^{i\theta}$ is a fixed element of $\partial D(0,1)$ that avoids $S$.  If one applies this bound for some $\theta$ close to $\pm \pi/2$ for either choice of sign $\pm$, we conclude that
$$ \pm \Im \mu \leq O( |\mu|^2 + \sigma^2 ) + O( (\eps+o(1))^n ).$$
Putting all this together, we conclude that
$$ \mu, 1-a \lesssim |\mu|^2 + \sigma^2 + (\eps+o(1))^n.$$
Since $|\mu|^2 = o( |\mu|)$, we can delete the $|\mu|^2$ term on the right-hand side.  This concludes the proof of (i).

Now we prove (ii). If $e^{i\theta} \in I$, then by Theorem \ref{limit-thm}(ii) there exists a zero $\zero$ of $f$ that lies within $o(1)$ of $e^{i\theta}$; it also lies inside the disk $\overline{D(0,1)}$, thus one may write $\zero = e^{i\theta} (1 - \eta)$ where $\eta=o(1)$ is such that $|1-\eta| \leq 1$.  From \eqref{exi} we have
\begin{equation}\label{back}
U_\deriv(a) - U_\deriv(e^{i\theta} - e^{i\theta} \eta) = o(\sigma^2) + O( (\eps+o(1))^n ).
\end{equation}
By the fundamental theorem of calculus one has
$$\log \frac{1}{|e^{i\theta} - e^{i\theta} \eta - \deriv|} - \log \frac{1}{|e^{i\theta}-\deriv|} = \log |1 - e^{-i\theta} \deriv| - \log |1 - \eta - e^{-i\theta} \deriv| = \Re \int_0^1 \frac{\eta dt}{1-\eta t -e^{-i\theta} \deriv}.$$
We can Taylor expand
$$ \frac{1}{1-\eta t -e^{-i\theta} \deriv} = \frac{1}{1-\eta t} + \frac{e^{-i\theta} \deriv}{(1- \eta t)^2} + O( |\deriv|^2 )$$
hence on taking expectations and simplifying using \eqref{var-ident}, \eqref{meao} and $\eta=o(1)$
$$  U_\deriv(e^{i\theta} - e^{i\theta} \eta) - U_\deriv(e^{i\theta}) = \Re \eta \int_0^1 \frac{dt}{1-\eta t} + o( \sigma^2 ) +
 O\left( (\eps+o(1))^n \right).$$
Replacing $\deriv$ by $0$, we also have
$$  U_0(e^{i\theta} - e^{i\theta} \eta) - U_0(e^{i\theta}) = \Re \int_0^1 \frac{\eta dt}{1-\eta t}.$$
Since $U_0(e^{i\theta})=0$ and $U_i(e^{i\theta} - e^{i\theta} \eta) = \log \frac{1}{|1-\eta|} \geq 0$, we thus have
$$  U_\deriv(e^{i\theta} - e^{i\theta} \eta) - U_\deriv(e^{i\theta}) \geq - o( \sigma^2 ) -
 O\left( (\eps+o(1))^n \right).$$
Combining this with \eqref{back}, we conclude that
$$ U_\deriv(a)-U_\deriv(e^{i\theta}) \geq  - o(\sigma^2) - (\eps+o(1))^n$$
and the claim (ii) follows from the underspill principle.

Now we prove (iii).  Suppose first that $\varphi$ is non-negative.  The fixed set $S$ only intersects $\partial D(0,1)$ in at most finitely many points.  Thus, for any fixed $0 < \delta \leq 1$, one can find a fixed compact subset $I$ of $\partial D(0,1) \backslash S$ whose complement in $\partial D(0,1)$ has measure at most $\delta$ with respect to the uniform probability measure $dm$ on $\partial D(0,1)$.  Integrating (ii) we see that if $n$ is sufficiently large depending on $\delta, \eps$ then
\begin{equation}\label{sin}
 \left(\int_I \varphi\ dm\right) U_\deriv(a) - \int_I  \varphi(e^{i\theta}) U_\deriv(e^{i\theta})\ dm(e^{i\theta}) \geq - o(\sigma^2) - o(1)^n.
\end{equation}
Now consider the expression
$$ \int_{\partial D(0,1) \backslash I}  \varphi(e^{i\theta}) \left(\log |e^{i\theta} - \deriv| + \Re(e^{-i\theta} \deriv)\right)\ dm(e^{i\theta}).$$
On the one hand, from Cauchy-Schwarz and the uniform square-integrability of the functions $e^{i\theta} \mapsto \log |e^{i\theta} - \deriv| + \Re (e^{-i\theta} \deriv)$ we can bound this quantity by $O(\delta^{1/2})$ (where the implied constant can depend on $\varphi$, but is independent of $\delta$).  On the other hand, if $|\deriv| \leq 1/2$ (say), we can perform a Taylor expansion
$$ \log |e^{i\theta} - \deriv| + \Re(e^{-i\theta} \deriv) = O( |\deriv|^2)$$
and we can bound this quantity instead by $O( \delta |\deriv|^2 )$.  Thus for all $\deriv \in \overline{D(0,1)}$ we can obtain a bound of $O(\delta^{1/2} |\deriv|^2)$.  Taking expectations and using \eqref{meao} to dispose of the $\Re e^{i\theta} \deriv$ contribution we conclude that
$$ \int_{\partial D(0,1) \backslash I} \varphi(e^{i\theta}) U_\deriv(e^{i\theta}) \ dm(e^{i\theta})
= O( \delta^{1/2} \sigma^2 ) + o(\sigma^2) + o(1)^n.$$
Also we certainly have
$$ \int_{\partial D(0,1) \backslash I} \varphi\ dm = O(\delta).$$
Combining these estimates with \eqref{sin}, \eqref{abar}, we conclude that
\begin{align*}
& \left(\int_{\partial D(0,1)} \varphi\ dm\right) U_\deriv(a) -  \int_{\partial D(0,1)}\varphi(e^{i\theta}) U_\deriv(e^{i\theta}) \ dm(e^{i\theta}) \\
&\quad \geq  - O\left( \delta^{1/2} \sigma^2 \right) - o\left(\sigma^2\right) - o(1)^n
\end{align*}
where the implied constant in the $O\left( \delta^{1/2} \sigma^2 \right)$ term is independent of $\delta$.  Sending $\delta$ sufficiently slowly to zero, we conclude that
\begin{equation}\label{bing}
\int_{\partial D(0,1)}\varphi(e^{i\theta}) (U_\deriv(a) - U_\deriv(e^{i\theta})) \ dm(e^{i\theta}) \geq - o(\sigma^2) - o(1)^n. 
\end{equation}
If we apply \eqref{bing} with $\varphi$ identically equal to one, we see from Jensen's formula (or \eqref{Taylor}) that the contribution of the $U_\deriv(e^{i\theta})$ term vanishes, giving 
$$ - U_\deriv(a) \leq o(\sigma^2) + o(1)^n.$$
From Lemma \ref{basic}(i) the left-hand side is positive, and we obtain \eqref{abar}.  (Compare this with the argument used to rule out \eqref{lamin}.)  From \eqref{bing} we now have
$$
 \int_{\partial D(0,1)} \varphi(e^{i\theta}) U_\deriv(e^{i\theta})\ dm(e^{i\theta}) \leq o(\sigma^2)  + o(1)^n 
$$
whenever $\varphi\colon \partial D(0,1) \to \C$ is fixed smooth and non-negative.  Since one can add an arbitrary constant to $\varphi$ without affecting the left-hand side by Jensen's formula (or \eqref{Taylor}), this inequality in fact holds for all fixed smooth real $\varphi$; replacing $\varphi$ with $-\varphi$ if necessary we now obtain \eqref{bbar} for real $\varphi$, and the complex case now follows by the triangle inequality.
\end{proof}

We can now finish the proof of Theorem \ref{contra-circle}, following the ``second Fourier coefficient'' strategy used to rule out \eqref{dd}.  From Taylor expansion and Fourier inversion one has
$$
 \int_0^{2\pi} e^{-2i\theta} \log \frac{1}{|e^{i\theta} - \deriv|}  \ \frac{d\theta}{2\pi}= \frac{1}{4} \deriv^2$$
for any $\deriv \in \overline{D(0,1)}$ (one can work first with $\deriv \in D(0,1)$ and then take limits if desired).  Applying Proposition \ref{xi-fine}(iii) and the Fubini--Tonelli theorem, we conclude that
$$ \E \deriv^2 = o(\sigma^2) + o(1)^n$$
and in particular
\begin{equation}\label{rexi2}
 \E \Re(\deriv^2) = o(\sigma^2) + o(1)^n
\end{equation}
To derive a contradiction from this,  observe that as $1 \leq |a-\deriv| \leq 2$, we surely have
$$ |a-\deriv| - 1 \sim \log |a-\deriv|$$
and hence by the triangle inequality
$$ |1-\deriv| = 1 + O( (1-a) + \log |a-\deriv| ).$$
Thus $\deriv$ lies within $O((1-a) + \log |a-\deriv|)$ of an element $\tilde \deriv$ of the arc $\overline{D(0,1)} \cap D(1,1)$.  In particular, using once again the crucial observation from Figure \ref{fig:key}, the argument of $\tilde \deriv$ lies in either $[\pi/3,\pi/2]$ or $[-\pi/2, -\pi/3]$, so that 
$$ \Re(\tilde \deriv^2) \leq - \frac{1}{2} |\tilde \deriv|^2.$$
Writing $\tilde \deriv$ in terms of $\deriv$ and an error of size $O((1-a) + \log |a-\deriv|)$, we conclude that
$$ \Re(\deriv^2) \leq - \frac{1}{2} |\deriv|^2 + O\left( |\deriv| \left((1-a) + \log |a-\deriv|\right) \right) + O\left( \left((1-a) + \log |a-\deriv|\right)^2 \right).$$
Using the Young inequality $|zw| \leq \frac{1}{2} |z|^2 + \frac{1}{2} |w|^2$ for various choices of $z,w$ to absorb some error terms into other terms, we conclude that
$$ \Re(\deriv^2) \leq - \frac{1}{4} |\deriv|^2 + O( (1-a)^2 + \log^2 |a-\deriv| )$$
(say). Thus on taking expectations
$$ \E \Re(\deriv^2) \leq - \frac{1}{4} \E |\deriv|^2 + O( (1-a)^2 + \E \log^2 |a-\deriv| ).$$

From \eqref{1a} and the bound $1-a=o(1)$ we have
$$ (1-a)^2 \leq o( \sigma^2 ) + (\eps+o(1))^n$$
for any fixed $\eps>0$, hence by the underspill principle
\begin{equation}\label{fint}
 (1-a)^2 \leq o( \sigma^2 ) + o(1)^n.
\end{equation}
Similarly, from \eqref{abar} and the bound $0 \leq \log |a-\deriv| \lesssim 1$ we have
$$ \E \log^2 |a-\deriv| \leq o( \sigma^2 ) + o(1)^n.$$
Using \eqref{var-ident}, we conclude that
$$ \E \Re(\deriv^2) \leq - \frac{1+o(1)}{4} \E |\deriv|^2 + o(1)^n$$
which when combined with \eqref{rexi2}, \eqref{var-ident} gives
\begin{equation}\label{xin}
 |\mu|^2 + \sigma^2 = \E |\deriv|^2 = o(1)^n.
\end{equation}
Applying \eqref{fint} we conclude that
$$ a = 1 - o(1)^n.$$
But this contradicts \eqref{hypo}, and thus concludes the proof of Theorem \ref{contra-circle}.

\begin{remark}\label{miller-time} As an alterative to contradicting \eqref{hypo}, one could now conclude from \eqref{xin} (and the fact that there are only $n-1$ zeroes of $f'$) that $\deriv = o(1)^n$ \emph{surely} holds.  At this point one could repeat the arguments of Miller \cite{miller} or V\^aj\^aitu--Zaharescu \cite{vz}, carefully keeping track of the dependence of constants on $n$, to obtain a contradiction (basically the point is that all the implied constants in the arguments in \cite{miller} can be verified to be of exponential size $O(1)^n$ in $n$).  This can be used as a slightly simpler alternative to using the arguments of Chijiwa \cite{chijiwa} or Kasmalkar \cite{kasmalkar} to cover the remaining $a^{(\infty)}=1$ cases of Theorem \ref{main-contr}.  We leave the details of this variant to the interested reader.
\end{remark}

\begin{remark}  Suppose one only excludes zeroes of $f$ in $D(a,1)$ rather than $\overline{D(a,1)}$.  Then one can check that the arguments in Theorem \ref{contra-origin} still hold as long as $a \neq 0$, and the arguments in Theorem \ref{contra-circle} are completely unaffected.  A careful reading of the arguments in \cite{degot}, \cite{chalebgwa} reveals that the ranges of Theorem \ref{main-contr} covered by those results also continue to be valid with this weaker hypothesis.  As mentioned in the introduction, the case $a=0$ under this hypothesis was covered by Schmeisser \cite{schmeisser}.  Finally, the results in \cite{miller}, \cite{vz}, \cite{chijiwa}, \cite{kasmalkar} also apply in this setting except when $a=1$, and in the $a=1$ case the results of Rubinstein \cite{rubinstein} shows that equality holds only for the polynomial $z^n-1$.  Putting this all together and undoing the normalisations that $f$ be monic and $\lambda_0$ be real, we see that for sufficiently large degrees $n$, the only case in which Sendov's conjecture can fail with $\overline{D(\lambda_0,1)}$ replaced by $D(\lambda_0,1)$ is if $f(z) = c (z^n - e^{i\theta})$ for some $c \in \C \backslash \{0\}$ and $\theta \in \R$.  This is consistent with a conjecture of Phelps and Rodriguez \cite{phelps}.
\end{remark}

\end{document}